\documentclass[leqno]{amsart}
\usepackage{amsfonts,amssymb,amsmath,amsgen,amsthm}
\usepackage{hyperref,color}

\theoremstyle{plain}
\newtheorem{theorem}{Theorem}[section]

\newtheorem{lemma}[theorem]{Lemma}
\newtheorem{corollary}[theorem]{Corollary}
\newtheorem{proposition}[theorem]{Proposition}
\newtheorem{hyp}[theorem]{Assumption}
\theoremstyle{remark}
\newtheorem{remark}[theorem]{Remark}

\newtheorem{example}[theorem]{Example}

\def\R{{\mathbf R}}
\def\N{{\mathbf N}}
\def\Sch{{\mathcal S}}
\def\O{\mathcal O}
\def\F{\mathcal F}
\def\L{\mathcal L}

\def\({\left(}
\def\){\right)}
\def\<{\left\langle}
\def\>{\right\rangle}
\def\le{\leqslant}
\def\ge{\geqslant}

\def\bv{{\bf v}}
\def\1{{\mathbf 1}}

\def\d{{\partial}}
\def\eps{\varepsilon}

\def\si{{\sigma}}

\def\Dt{{\Delta t}}
\def\E{{\mathcal E}}
\DeclareMathOperator{\RE}{Re}
\DeclareMathOperator{\IM}{Im}
\DeclareMathOperator{\DIV}{div}

\numberwithin{equation}{section}

\begin{document}

\title[Time splitting for semi-classical NLS]{On Fourier time-splitting methods
  for nonlinear Schr\"odinger equations
  in the semi-classical limit}
\author[R. Carles]{R\'emi Carles}
\address{CNRS \& Univ. Montpellier~2\\Math\'ematiques
\\CC~051\\34095 Montpellier\\ France}
\email{Remi.Carles@math.cnrs.fr}
\begin{abstract}
We prove an error estimate for a Lie-Trotter splitting
operator associated to the Schr\"odinger-Poisson equation in the
semiclassical regime, when the WKB approximation is
valid. In finite time, and so
long as the solution to a 
compressible Euler-Poisson equation is smooth, the
error between the numerical solution and the exact solution is
controlled in Sobolev spaces, in a suitable phase/amplitude
representation. As a corollary, we infer the numerical 
convergence of the quadratic observables with a time step independent
of the Planck constant. A similar result is
established for the nonlinear Schr\"odinger equation in the weakly
nonlinear regime. 
\end{abstract}
\thanks{This work was supported by the French ANR projects
  BECASIM (ANR-12-MONU-0007-04) and SchEq (ANR-12-JS01-0005-01).}
\maketitle

\section{Introduction}
\label{sec:intro}

We consider the nonlinear Schr\"odinger equation, for $t\ge 0$,
\begin{equation}
  \label{eq:nls}
  i\eps \d_t u^\eps +\frac{\eps^2}{2}\Delta u^\eps =
 \eps^\alpha f\(|u^\eps|^2\) u^\eps.
\end{equation}
The function
$u^\eps=u^\eps(t,x)$ is complex-valued, and the space variable $x$
belongs to $\R^d$. The presence of the
parameter $\eps$ is motivated by the \emph{semi-classical limit},
$\eps\to 0$. Physically, $\eps$ corresponds to a small ratio between
microscopic and macroscopic quantities, so the limit
  $\eps \to 0$ is expected to
yield a relevant approximation; see e.g. \cite{JiMaSp11} and
references therein. 
The parameter $\alpha\ge 0$ measures the strength
of nonlinear interactions: in the WKB regime, which is recalled below,
the nonlinearity is negligible if $\alpha>1$, it has a leading order
(moderate) influence if $\alpha=1$ (weakly nonlinear regime), and its
influence is very strong in the regime $\eps\to 0$ if $\alpha=0$. The
case $0<\alpha<1$ is not considered here, but it should be considered
as similar to the case $\alpha=0$ (\cite{CaBook}). 
\smallbreak

In this paper, we consider mostly two families of nonlinearity:
\begin{itemize}
\item Nonlocal nonlinearity in the case $\alpha=0$: $f(\rho) = K\ast
  \rho$.
\item Local or nonlocal nonlinearity in the case $\alpha\ge 1$. 
\end{itemize}
The first case includes the Schr\"odinger-Poisson system in space
dimension $d\ge 3$ ($f(\rho)=\lambda\Delta^{-1}\rho$, $\lambda \in
\R$, hence
$K(x)=\lambda c_d/|x|^{d-2}$). The second case includes 
the cubic nonlinearity (focusing or defocusing). We will also discuss
 why the case of the cubic nonlinearity is not
treated in the regime $\alpha=0$ (see Remark~\ref{rem:loss}). 
Our analysis is limited to bounded time intervals, so
  the exponentials are controlled, which is the reason why we do not
  keep track of such factors when they are eventually included in a
  uniform constant. The reason why the analysis is 
  bound to finite time intervals is, in the case $\alpha=0$, that the
  solution of the 
  Euler-Poisson equation generically develops a singularity in finite time,
  and, in the case $\alpha\ge 1$, that the solution of the
  Burgers' equation generically develops a singularity in finite time.
\smallbreak

The initial data that we consider are of WKB type:
\begin{equation}
  \label{eq:ci}
  u^\eps(0,x) = a_0(x)e^{i\phi_0(x)/\eps}. 
\end{equation}
An important well-known property of this framework is related to the
following quantities (quadratic observables), 
\begin{align*}
  \text{Position density: }& \rho^\eps(t,x)=|u^\eps(t,x)|^2.\\
\text{Current density: }& J^\eps(t,x) = \eps \IM\(\overline
u^\eps(t,x)\nabla u^\eps(t,x)\).
\end{align*}
Consider the case of Schr\"odinger-Poisson system in dimension $d\ge
3$, with $\alpha=0$. 
Formally,  $\rho^\eps$ and $J^\eps$
converge to the solution of the compressible Euler-Poisson equation
\begin{equation}
  \label{eq:euler1}
  \left\{
    \begin{aligned}
      &\d_t \rho +\DIV J =0 ;\quad \rho_{\mid t=0} = |a_0|^2,\\
&\d_t J + \DIV \(\frac{J\otimes J}{\rho}\) + \rho \nabla P=0;\quad
J_{\mid t=0} = |a_0|^2\nabla \phi_0,\\
&\Delta P = \lambda \rho,\quad P(t,x),\nabla P(t,x)\to 0\text{ as
}|x|\to \infty.  
    \end{aligned}
\right.
\end{equation}
 See e.g. \cite{BrezziMarkowich,ZhangSIMA} for a rigorous statement of
 this result.  
\subsection{Fourier time-splitting methods}
\label{sec:split}

When simulating numerically \eqref{eq:nls}, the size of $\eps$ becomes
an important parameter: if the nonlinearity $f$ is
replaced by an 
external potential $V(x)$ (independent of $u^\eps$), then it was
proved in \cite{MPP} that finite difference approximation requires to
consider a time step $\Delta t=o(\eps)$ in order to recover the above
quadratic observables. In \cite{BJM1}, it was proved that these
quadratic observables can be accurately recovered for time steps
independent of $\eps$, if time splitting methods are
considered and $V$ is bounded as well as all its
derivatives. Moreover, if $\Delta t =o(\eps)$, then the wave function 
$u^\eps$ itself is well approximated; see also
\cite[Theorem~2]{DeTh10}. In the appendix, we extend this result to
the case of unbounded potentials, which grow at most quadratically in
space. 
\smallbreak

In the nonlinear framework
\eqref{eq:nls}, numerical experiences in \cite{BJM2} suggest that
considering $\Delta t = \O(\eps)$ is enough to recover the correct
observables for time splitting spectral methods, when
$\alpha=1$, or $\alpha=0$ with a defocusing nonlinearity. In the
references mentioned so far, space 
discretization is considered too: in the present paper, we shall
discuss only the time discretization, hence the above restrictions.   
In the recent paper \cite{DeTh13}, some precise local error estimates
have been established, showing that the assumption $\Delta t =
\O(\eps)$ is a sensible assumption for the local error to behave
properly. We underscore that if crucial, the local error estimate is
not sufficient to obtain a global error estimate, unlike in
\cite{Lu08}, because of rapid oscillations. 
\smallbreak

We now briefly recall what time splitting methods consist in, in the
context of \eqref{eq:nls}. The remark is that if the Laplacian or the
nonlinearity is discarded in \eqref{eq:nls}, then the equation becomes
explicitly solvable. We denote by $X^t_\eps$ the map
$v^\eps(0,\cdot)\mapsto v^\eps(t,\cdot)$, where
\begin{equation}\label{eq:linear}
  i\eps\d_t v^\eps +\frac{\eps^2}{2}\Delta v^\eps=0. 
\end{equation}
The above equation
is solved explicitly by using the Fourier transform (defined in
Assumption~\ref{hyp:r3} below),
since it becomes an ordinary differential equation
\begin{equation}\label{eq:X}
  i\eps \d_t \widehat v^\eps -\frac{\eps^2}{2}|\xi|^2 \widehat v^\eps =0,
\end{equation}
hence
\begin{equation*}
  \widehat {X^t_\eps v}(\xi)=e^{-i\eps \frac{t}{2}|\xi|^2}\widehat v(\xi). 
\end{equation*}
If we now denote by $Y^t_\eps$ the map
$w^\eps(0,\cdot)\mapsto w^\eps(t,\cdot)$, where
\begin{equation}\label{eq:Y}
  i\eps\d_t w^\eps = \eps^\alpha f\(|w^\eps|^2\)w^\eps, 
\end{equation}
then we remark that since $f$ is real-valued, the modulus of $w^\eps$
does not depend on time, hence
\begin{equation}\label{eq:Yexpl}
  Y^t_\eps w(x) = w(x) e^{-i\eps^\alpha \frac{t}{\eps} f(|w(x)|^2)}.
\end{equation}
At this stage, it is already clear that whether $\alpha \ge 1$ or
$\alpha<1$, the estimates for $Y^t_\eps$ will be rather different. 
We shall denote by $S^t_\eps$ the nonlinear flow associated to
\eqref{eq:nls}: $S^t_\eps u^\eps(0,\cdot) = u^\eps(t,\cdot)$. 
\smallbreak

We consider the Lie-type splitting operator
\begin{equation}\label{eq:Z}
  Z_{\eps}^t =Y^t_\eps X^t_\eps ,
\end{equation}
for which calculations will be less involved than for the Strang-type
splitting operator 
\begin{equation*}
  Z_{\eps,S}^t= X_\eps^{t/2}Y_\eps^t X_\eps^{t/2}.
\end{equation*}
Since both $X_\eps^t$ and $Y_\eps^t$ are unitary on $L^2$, so is
$Z_\eps^t$:
\begin{equation}
  \label{eq:unitL2}
  \|X_\eps^t\|_{L^2\to L^2} = \|Y_\eps^t\|_{L^2\to
    L^2}=\|Z_\eps^t\|_{L^2\to L^2}=1. 
\end{equation}
The action of $Z_\eps^t$ on Sobolev spaces is more involved, because
of the nonlinear operator $Y_\eps^t$ (in the case $\alpha=0$).
In the case $\eps=1$ with $f(y)=y$ (cubic nonlinearity), the
convergence of the approximate solution 
generated by the splitting operator as the time step goes to zero has
been established in \cite{BBD} for $x\in \R^d$, $d\le 2$, and in \cite{Lu08} for
$x\in \R^3$.
\begin{theorem}[From \cite{BBD,Lu08}]
  Let $\eps=1$, $f(y)=y$, and $d\le 2$. For all $u_0=u^\eps_{\mid t=0}\in
  H^2(\R^d)$ and all $T>0$, there exist $C$ and $h_0$ such that for
  all $\Delta t\in (0,h_0]$, for all $n\in \N$ such that $n\Delta t\in [0,T]$, 
  \begin{equation*}
    \left\| \(Z^{\Delta t}_1\)^n u_0 -S^{n\Delta
        t}u_0\right\|_{L^2}\le C\(m_2,T\) \Delta t, 
  \end{equation*}
where, for $j\in \N$,
\begin{equation*}
  m_j= \max_{0\le t\le T} \|u(t)\|_{H^j(\R^d)}.
\end{equation*}
If $d=3$ and  $u_0\in H^4(\R^d)$, then 
 \begin{equation*}
    \left\| \(Z^{\Delta t}_{1,S}\)^n u_0 -S^{n\Delta
        t}u_0\right\|_{L^2}\le C\(m_4,T\) (\Delta t)^2. 
  \end{equation*}
\end{theorem}
 Note
however that these results do not directly yield interesting
information 
in the case of \eqref{eq:nls} in the semi-classical limit: in the presence
of rapid oscillations as in \eqref{eq:nls}, the quantity $m_j$ behaves
like $\eps^{-j}$, so the bounds in  \cite{BBD,Lu08} cease to be
interesting. 
\smallbreak

On a more technical level, note that even though $Z_\eps^t$ is
unitary on $L^2$, a standard Lady
Windermere's fan argument, which 
consists in writing 
\begin{equation}\label{eq:lady1}
   u_n - u(t_n) = \sum_{j=0}^{n-1} \(\(Z_\eps^\Dt\)^{n-j-1}Z_\eps^\Dt 
S_\eps^{j\Dt}u_0  - \(Z_\eps^\Dt\)^{n-j-1}S_\eps^\Dt S_\eps^{j\Dt}u_0\), 
\end{equation}
cannot be used directly, since $Z_\eps^t$ is not a linear
operator. Therefore, nonlinear estimates are needed.
Eventually, a Lady
Windermere's fan argument different from \eqref{eq:lady1} is used.
In the case of the Schr\"odinger-Poisson system, the proof in
\cite{Lu08} uses for instance the estimate
\begin{equation*}
  \|\Delta^{-1}(uv)w\|_{L^2(\R^3)}\le C
  \|u\|_{H^1(\R^3)}\|v\|_{L^2(\R^3)}\|w\|_{L^2(\R^3)}. 
\end{equation*}
In the present framework, functions are $\eps$-oscillatory (see
Remark~\ref{rem:osc} below), so the
natural adaptation of the above estimates is of the form
\begin{equation*}
  \|\Delta^{-1}(u^\eps v^\eps) w^\eps\|_{L^2(\R^3)}\le C
  \eps^{-1/2}\|u^\eps\|_{H^1_\eps(\R^3)}\|v^\eps\|_{L^2(\R^3)}
\|w^\eps\|_{L^2(\R^3)}, 
\end{equation*}
where $C$ is independent of $\eps$ and
\begin{equation*}
  \|u^\eps\|_{H^1_\eps(\R^3)} = \sup_{0<\eps\le 
      1}\(\|u^\eps\|_{L^2}+\|\eps\nabla u^\eps \|_{L^2}\)
\end{equation*}
is expected to be bounded uniformly in $\eps$, unlike the standard
$H^1$-norm. 
We  then face an $\eps^{-1/2}$ singular factor in the above
estimate, which ruins the approach of \cite{Lu08} in the
semi-classical limit. Such phenomena explain why there is a gap
between the proof in the semi-classical regime for the \emph{linear}
Schr\"odinger equation \cite{DeTh10} and adapting the arguments of
\cite{Lu08} to the semi-classical regime, even with the local error
estimate of \cite{DeTh13}. 
\subsection{WKB analysis}
\label{sec:wkb}

Given \eqref{eq:nls} with initial datum \eqref{eq:ci}, WKB method
consists in seeking
\begin{equation*}
  u^\eps(t,x) = a^\eps(t,x)e^{i\phi(t,x)/\eps},\quad \text{with
  }a^\eps \approx a+\eps a^{(1)}+\dots
\end{equation*}
Plugging this ansatz into \eqref{eq:nls} and ordering the powers of
$\eps$, we find formally:
\begin{equation*}
  \O(\eps^0): \quad \d_t \phi + \frac{1}{2}|\nabla \phi|^2 = 
\left\{
  \begin{aligned}
   & 0 &\text{ if }\alpha \ge 1,\\
&-f\(|a|^2\)& \text{ if }\alpha=0.
  \end{aligned}
\right.
\end{equation*}
\begin{equation*}
  \O(\eps^1): \quad \d_t a + \nabla \phi\cdot \nabla a
  +\frac{1}{2}a\Delta \phi = 
\left\{
  \begin{aligned}
  &  0 &\text{ if }\alpha >1,\\
&-if\(|a|^2\)a  &\text{ if }\alpha =1,\\
&-2i f'\(|a|^2\)a\RE\(\overline a a^{(1)}\) & \text{ if }\alpha=0.
  \end{aligned}
\right.
\end{equation*}
We see that if $\alpha>1$, then the nonlinearity does not affect the
pair $(a,\phi)$, which describes the behavior of $u^\eps$ at leading
order. On the other hand, if $\alpha=1$, the transport equation for
$a$ is nonlinear, while the equation for $\phi$ is the same as in the
linear case: one speaks of \emph{weakly nonlinear} regime. Finally, in
the case $\alpha=0$, 
the system of equations shows a strong coupling
between all the terms, and is actually not even closed. 
\smallbreak
In the rest of this subsection, we focus our attention on the case
$\alpha=0$. 
An important remark consists in noticing that the
transport equation
\begin{equation}\label{eq:transportsurcrit}
  \d_t a + \nabla \phi\cdot \nabla a
  +\frac{1}{2}a\Delta \phi = -2i f'\(|a|^2\)a\RE\(\overline a
  a^{(1)}\), 
\end{equation}
enjoys the following
property: even though it cannot be solved when $a^{(1)}$ is unknown,
it is of the form $D_t a = ia\times \R$, where $D_t$ stands 
for the vector field $\d_t +\nabla \phi\cdot \nabla +\frac{1}{2}\Delta
\phi$. Therefore, $D_t|a|^2=0$, and if we set $(v,\rho)=(\nabla
\phi,|a|^2)$, then the system in $(\phi,a)$ becomes the closed system
\begin{equation}
  \label{eq:euler}
    \left\{
    \begin{aligned}
      &\d_t \rho +\DIV (\rho v) =0 ;\quad \rho_{\mid t=0} = |a_0|^2,\\
&\d_t v+v\cdot \nabla v+ \nabla f(\rho)=0;\quad
v_{\mid t=0} = \nabla \phi_0.
    \end{aligned}
\right.
\end{equation}
Note also that if we set $\tilde J = \rho v$, then $(\rho,\tilde J)$
solves \eqref{eq:euler1}: we have written \eqref{eq:euler1} in a
different form, which is also encountered in fluids mechanics.
As a matter of fact, in the case of a nonlocal nonlinearity
$f(\rho)=K\ast \rho$, \eqref{eq:transportsurcrit} is not correct, but
since this term has disappeared in \eqref{eq:euler}, we do not write
the correct version of \eqref{eq:transportsurcrit}, which is a bit
involved to present. In the case of a nonlocal nonlinearity, we will
make the following assumption. 
\begin{hyp}\label{hyp:r3}
  The nonlinearity $f$ is of the form $f(\rho)=K\ast \rho$, where the
  kernel $K$ is such that its Fourier transform, defined by
  \begin{equation*}
    \widehat K(\xi) = \frac{1}{(2\pi)^{d/2}}\int_{\R^d} e^{-ix\cdot
      \xi}K(x)dx, 
  \end{equation*}
satisfies:
\begin{itemize}
\item If $d\le 2$, 
  \begin{equation*}
    \sup_{\xi\in \R^d}(1+|\xi|^2)\lvert \widehat K(\xi)\rvert <\infty.
  \end{equation*}
\item If $d\ge 3$, 
\begin{equation*}
  \sup_{\xi\in \R^d}|\xi|^2\lvert \widehat K(\xi)\rvert <\infty.
\end{equation*}
\end{itemize}
\end{hyp}
Typically, this includes the case of Schr\"odinger-Poisson system if
$d\ge 3$,
where $f(\rho)$ is given by the Poisson equation
\begin{equation*}
  \Delta f = \lambda \rho,\quad f,\nabla f \to 0 \text{ as }|x|\to \infty,
\end{equation*}
with $\lambda \in \R$. This equation can be solved by Fourier
analysis if $d\ge 3$ ($\widehat K(\xi) = -\lambda |\xi|^{-2}$); if
$d\le 2$, this is no longer the case, as discussed in \cite{Mas10}. 
Under this assumption, \eqref{eq:euler} has a unique solution
$(v,\rho)\in C([0,T];H^{s+1} \times (H^s\cap L^1))$ provided that the initial
data are sufficiently smooth, with $s>d/2+1$, from \cite{Ga93} (see also
\cite{AC-SP,LiLinEJDE,ZhangSIMA}, and Section~\ref{sec:cadre} for the
main steps of the proof).

\begin{proposition}
Suppose that $f$ satisfies Assumption~\ref{hyp:r3}.   Let
$a_0,\phi_0\in \Sch'(\R^d)$ with $(\nabla \phi_0,a_0)\in 
 H^{s+1}\times H^s$ for some $s>d/2$. There exists a unique
 maximal solution  
   $(v,\rho)\in C\([0,T_{\rm max});H^{s+1}\times (H^s\cap L^1)\)$
   to \eqref{eq:euler}.
    In addition,  $T_{\rm max}$ is independent of $s>d/2+1$
  and 
    \begin{equation*}
   T_{\rm max}<+\infty\Longrightarrow  \int_0^{T_{\rm
       max}}\(\|v(t)\|_{W^{1,\infty}}+\|a(t)\|_{W^{1,\infty}}\)dt=+\infty. 
  \end{equation*}
\end{proposition}
\begin{remark}\label{rem:osc}
We note from \eqref{eq:euler} that even if no rapid oscillation is
present initially in \eqref{eq:ci}, then  $v_{\mid t=0}=0$ and $\d_t
v_{\mid t=0}\not =0$, so the solution $u^\eps$ is not
$\eps$-oscillatory at time $t=0$, but becomes \emph{instantaneously}
$\eps$-oscillatory. 
\end{remark}

We emphasize the fact that under Assumption~\ref{hyp:r3}, and for fixed
$\eps>0$, given
$u_0^\eps\in L^2(\R^d)$, \eqref{eq:nls} has a unique, global solution
$u^\eps\in C([0,\infty);L^2)$. Moreover, higher Sobolev regularity is
propagated globally in time
(the nonlinearity is $L^2$-subcritical); see e.g. \cite{CazCourant}.

\subsection{Main results}
\label{sec:main}

Our main result measures the accuracy of the time splitting operator
so long as the solution to \eqref{eq:euler} remains smooth.

\begin{theorem}\label{theo:main}
 Suppose that $d\ge 1$, $\alpha=0$ in \eqref{eq:nls}, and 
  that $f$ satisfies Assumption~\ref{hyp:r3}. Let
  $(\phi_0,a_0)\in L^\infty(\R^d)\times H^s(\R^d)$ with $s>d/2+2$,
  and such that $\nabla \phi_0 \in H^{s+1}(\R^d)$. 
Let $T>0$ be such that the solution to 
  \eqref{eq:euler} satisfies $(v,\rho)\in C([0,T];H^{s+1}\times
  H^s)$. Consider $u^\eps =S^t_\eps u_0^\eps$ solution to
  \eqref{eq:nls} and $u_0^\eps$ given by
  \eqref{eq:ci}. 
There exist $\eps_0>0$ and $C,c_0$ independent of $\eps\in (0,\eps_0]$
such that for 
  all $\Dt\in (0,c_0]$, for all $n\in \N$ such that $t_n=n\Dt\in
  [0,T]$, the following holds:\\ 
$1.$ There exist $\phi^\eps$ and $a^\eps$ with
\begin{equation*}
  \sup_{t\in [0,T]}\(\|a^\eps(t)\|_{H^s(\R^d)}+\|\nabla
  \phi^\eps(t)\|_{H^{s+1}(\R^d)}+\|\phi^\eps(t)\|_{L^\infty(\R^d)} \)\le 
  C,\quad \forall 
  \eps\in (0,\eps_0],
\end{equation*}
such that $u^\eps(t,x)=a^\eps(t,x)e^{i\phi^\eps(t,x)/\eps}$ for all
$(t,x)\in [0,T]\times \R^d$.\\
$2.$ There exist $\phi_n^\eps$ and $a_n^\eps$ with
\begin{equation*}
  \|a_n^\eps\|_{H^s(\R^d)}+\|\nabla\phi^\eps_n\|_{H^{s+1}(\R^d)}+
\|\phi^\eps_n\|_{L^\infty(\R^d)}\le C,\quad \forall
  \eps\in (0,\eps_0],
\end{equation*}
such that $(Z_\eps^\Dt)^n \(a_0 e^{i\phi_0/\eps}\) = a_n^\eps
e^{i\phi_n/\eps}$, and the following  error estimate holds:
  \begin{equation*}
 \left\| a_n^\eps - a^\eps(t_n)\right\|_{H^{s-1}}+ \left\| \nabla\phi_n^\eps -
   \nabla\phi^\eps(t_n)\right\|_{H^{s}}+ \left\| \phi_n^\eps -
   \phi^\eps(t_n)\right\|_{L^\infty}\le C \Dt.  
 \end{equation*}
\end{theorem}
Note that in the above result, the phase/amplitude representation of
the exact solution $u^\eps$ and the numerical solution is not unique. 
This result shows in particular that the splitting solution remains
bounded in $L^\infty$, uniformly in $\eps$, in the WKB regime. 
We infer the convergence of 
the wave functions in $L^2$, by reconstructing the numerical wave function:
\begin{corollary}\label{cor:wave}
  Under the assumptions of Theorem~\ref{theo:main}, there exist
  $\eps_0>0$ and 
$C,c_0$ independent of $\eps\in (0,\eps_0]$ such that for
  all $\Dt\in (0,c_0]$, for all $n\in \N$ such that $n\Dt\in [0,T]$, 
 \begin{equation*}
 \left\| (Z_\eps^\Dt)^n u_0^\eps
   -S^{t_n}_\eps
      u_0^\eps\right\|_{L^2(\R^d)}\le C \frac{\Dt}{\eps}. 
 \end{equation*}
\end{corollary}
We also get the convergence of the main
quadratic observables:
\begin{corollary}\label{cor:CVquad}
  Under the assumptions of Theorem~\ref{theo:main}, there exist
  $\eps_0>0$ and 
$C,c_0$ independent of $\eps\in (0,\eps_0]$ such that for
  all $\Delta t\in (0,c_0]$, for all $n\in \N$ such that $n\Delta
  t\in [0,T]$,  
  \begin{align*}
   & \left\| \left\lvert (Z_\eps^\Dt)^n u_0^\eps\right\rvert^2
     -|\rho^\eps(t_n)|^2   \right\|_{L^1(\R^d)\cap L^\infty(\R^d)}\le C \Dt,\\ 
&\left\| \IM\(\eps \overline{(Z_\eps^\Dt)^n u_0^\eps}\nabla
  (Z_\eps^\Dt)^n u_0^\eps\) - J^\eps(t_n) \right\|_{L^1(\R^d)\cap
  L^\infty(\R^d)}\le C \Dt.  
 \end{align*}
\end{corollary}

These results may seem limited, inasmuch as they address only a
specific regime, and say nothing on the large time behavior. We
emphasize the fact that the behavior of $u^\eps$ as $\eps\to 0$ at
time where the solution to \eqref{eq:euler} ceases to be smooth is still
an open question. Therefore, the analytical tools to analyze the
splitting operators are missing, due to a lack of precise estimates on
the exact solution. Typically, all the results presented here highly
rely on the fact that a WKB regime is considered. 

\subsection{Weakly nonlinear regime}
\label{sec:wnlgointro}

We now consider the case $\alpha\ge 1$ in \eqref{eq:nls}, which turns out
to be quite easier to treat. To begin with, the assumption on the
nonlinearity is weaker, and we allow local interactions. 
\begin{hyp}\label{hyp:f}
  The nonlinearity $f$ is of the form $f=f_1+f_2$, where $f_1$
  satisfies Assumption~\ref{hyp:r3}, and $f_2\in
  C^\infty([0,\infty);\R_+)$, with $f_2(0)=0$. 
\end{hyp}
\begin{remark}
  The assumption $f_2(0)=0$ is here merely to
simplify the presentation, since replacing $f$ with $f-f_2(0)$ in
\eqref{eq:nls} amounts to replacing $u^\eps$ with $u^\eps
e^{itf_2(0)/\eps}$. 
\end{remark}

\begin{proposition}\label{prop:wnlgo}
  Suppose that $d\ge 1$, $f$ satisfies Assumption~\ref{hyp:f}, and
  that $\alpha \ge 1$ in \eqref{eq:nls}. Let $(\phi_0, a_0)\in 
  H^{s+2}\times H^s$ with $s>d/2+2$. Let $T>0$ be such that the solution to
  \begin{equation*}
    \d_t \phi+\frac{1}{2}|\nabla \phi|^2=0;\quad \phi_{\mid t=0}=\phi_0
  \end{equation*}
satisfies $\phi \in C([0,T];H^{s+2})$. Consider $u^\eps =S^t_\eps
u_0^\eps$ solution to 
  \eqref{eq:nls} with $\alpha\ge 1$ and $u_0^\eps$ given by
  \eqref{eq:ci}. 
There exist $\eps_0>0$ and $C,c_0$ independent of $\eps\in (0,\eps_0]$
such that for 
  all $\Dt\in (0,c_0]$, for all $n\in \N$ such that $n\Dt\in
  [0,T]$, the following holds:\\ 
$1.$ If we set $a^\eps = u^\eps e^{-i\phi/\eps}$, then 
\begin{equation*}
  \sup_{t\in [0,T]}\|a^\eps(t)\|_{H^s(\R^d)}
  \le 
  C,\quad \forall 
  \eps\in (0,\eps_0].
\end{equation*}
$2.$ There exist $\phi_n^\eps$ and $a_n^\eps$ with
\begin{equation*}
  \|a_n^\eps\|_{H^s(\R^d)}+\|\phi^\eps_n\|_{H^{s+2}(\R^d)}\le C,\quad \forall
  \eps\in (0,\eps_0],
\end{equation*}
such that $(Z_\eps^\Dt)^n \(a_0 e^{i\phi_0/\eps}\) = a_n^\eps
e^{i\phi_n/\eps}$, and the following  error estimate holds:
  \begin{equation*}
 \left\| a_n^\eps - a^\eps(t_n)\right\|_{H^{s-2}}+ \left\| \phi_n^\eps -
   \phi(t_n)\right\|_{H^{s}}\le C \Dt.  
 \end{equation*}
In particular,
  \begin{align*}
   \left\|(Z_\eps^\Dt)^n u_0^\eps -S^{n\Dt}_\eps
      u_0^\eps\right\|_{L^2}\le C \frac{\Dt}{\eps}.
 \end{align*}
\end{proposition}
It may seem surprising that even in the  weakly nonlinear
regime $\alpha = 1$, the result is local in time, and valid only
before the possible formation of caustics.
As a matter of fact, the behavior of the nonlinear solution $u^\eps$
is essentially not understood past the caustic; see e.g. \cite{CaBook}. 
\smallbreak

\noindent{\bf Notations.} Throughout the text, 
all the constants 
are independent of $\eps\in (0,1]$. For $(\alpha^\eps)_{0<\eps\le 1}$
and $(\beta^\eps)_{0<\eps\le 1}$  two families 
of positive real numbers, we write $\alpha^\eps \lesssim \beta^\eps$ if 
$\displaystyle \limsup_{\eps\to 0}\alpha^\eps/\beta^\eps <\infty$.

\section{Action of the numerical scheme in the WKB regime}
\label{sec:frame}

Our approach consists in sticking to
the WKB framework. We write the solutions to \eqref{eq:nls},
under the form
\begin{equation}\label{eq:bkw}
  a^\eps(t,x) e^{i\phi^\eps(t,x)/\eps},
\end{equation}
with $a^\eps$ and $\phi^\eps$ bounded in $H^s(\R^d)$
\emph{uniformly in} $\eps \in (0,1]$. Here, the ``phase'' $\phi^\eps$
is real-valued, and the ``amplitude'' $a^\eps$ is complex-valued; of
course, such a representation is not unique. 
As a matter of fact, both
$a^\eps$ and $\phi^\eps$ must be expected to depend on
$\eps$. Functions of this form will be referred to as \emph{WKB
  states} throughout the text. 
From now on, and up to Section~\ref{sec:wnlgo}, we assume $\alpha=0$
in \eqref{eq:nls}.

\subsection{A stable phase/amplitude decomposition}
\label{sec:decomp}

The important remark consists in noticing that the flows associated to
\eqref{eq:X} and \eqref{eq:Y} preserve the structure of WKB states. 

\subsubsection*{Nonlinear flow}

In the case of the nonlinear flow \eqref{eq:Y}, the exact formula
\eqref{eq:Yexpl} shows immediately that a WKB state evolves as a WKB state:
if $w^\eps_{\mid t=0} = \alpha^\eps e^{i\varphi^\eps/\eps}$, then the
solution to \eqref{eq:Y} is given by 
\begin{equation*}
  w^\eps(t,x) = \alpha^\eps(x)
  e^{i\(\varphi^\eps(x)-tf\(|\alpha^\eps(x)|^2\)\)/\eps}. 
\end{equation*}
This is indeed of the form \eqref{eq:bkw}, with
\begin{equation*}
  a^\eps(t,x) = \alpha^\eps(x),\quad \phi^\eps(t,x)
  =\varphi^\eps(x)-tf\(|\alpha^\eps(x)|^2\). 
\end{equation*}
We can therefore rewrite the action of $Y_\eps^t$ on WKB states as the
action of the flow 
${\mathcal Y}_t^\eps$ on phase/amplitude pairs $(\phi,a)$ characterized by
\begin{equation}
  \label{eq:Ysplit}
  \left\{
\begin{aligned}
      &\d_t \phi^\eps+f\(|a^\eps|^2\)=0 ;\quad
      \phi^\eps_{\mid t=0} =  
\phi_0^\eps,\\
&\d_t a^\eps = 0;\quad a^\eps_{\mid t=0}=a_0^\eps. 
    \end{aligned}
\right.
\end{equation}

\subsubsection*{Linear flow}
The analysis of the linear flow \eqref{eq:linear} is
less straightforward, and requires more care than the nonlinear
flow. Consider the system 
\begin{equation}
  \label{eq:Xbkw}
  \left\{
    \begin{aligned}
      &\d_t \phi^\eps+\frac{1}{2}|\nabla \phi^\eps|^2 =0 ;\quad
      \phi^\eps_{\mid t=0} = 
\phi_0^\eps,\\
&\d_t a^\eps +\nabla \phi^\eps\cdot \nabla a^\eps
+\frac{1}{2}a^\eps\Delta \phi^\eps = i\frac{\eps}{2}\Delta
a^\eps;\quad a^\eps_{\mid t=0}=a_0^\eps. 
    \end{aligned}
\right.
\end{equation}
Note that this is not exactly the system corresponding to standard WKB
analysis, because of the term 
$\eps\Delta a^\eps$ in the second equation, which is discarded in WKB
approximation. The first 
equation is an eikonal equation, which has a
smooth solution at least locally in time (see e.g. \cite{CaBook}), and
energy estimates then follow easily for the second equation. We
emphasize the fact that \eqref{eq:Xbkw} is \emph{equivalent} to
\eqref{eq:linear}
in the case of initial WKB states \eqref{eq:bkw}, at least locally in
time, \emph{modulo} the eikonal equation. Indeed, given an 
initial phase $\phi_0^\eps$, we can solve, locally in time, the eikonal equation
\begin{equation}
  \label{eq:eikonal}
  \d_t \phi^\eps+\frac{1}{2}|\nabla \phi^\eps|^2 =0 ;\quad
      \phi^\eps_{\mid t=0} = 
\phi_0^\eps.
\end{equation}
In general, the solution to \eqref{eq:eikonal} does not remain smooth
for all time, due to the formation of caustics (see
e.g. \cite{CaBook}). 
We note that
$w^\eps=\nabla \phi^\eps$ solves a (multidimensional) Burgers equation
\begin{equation*}
  \d_t w^\eps +w^\eps\cdot \nabla w^\eps=0;\quad w^\eps_{\mid
    t=0}=\nabla \phi_0^\eps. 
\end{equation*}
This remark will be used to derive \emph{a priori} estimates for the
system \eqref{eq:Xbkw}. Once $\phi^\eps$ is known, then $v^\eps$,
solution to \eqref{eq:linear}, and $a^\eps$, are related through the
formula
\begin{equation*}
  v^\eps=a^\eps e^{i\phi^\eps/\eps},
\end{equation*}
which yields an obvious bijective correspondence between these two
functions (for a fixed $\phi^\eps$). Even though
there is no uniqueness in the choice of $\phi^\eps$,
we conclude that if  
the initial datum is a WKB state, $v^\eps_{\mid t=0} = a_0^\eps
e^{i\phi_0^\eps/\eps}$, then at least locally in time, $v^\eps$ remains
a WKB state, since it can be written as $v^\eps = a^\eps
e^{i\phi^\eps/\eps}$, where $(\phi^\eps,a^\eps)$ is the solution to
\eqref{eq:Xbkw}. Following the same convention as in the case of the
nonlinear flow, we denote by ${\mathcal X}_\eps^t$ the flow acting of
phase/amplitude pairs,
\begin{equation*}
  {\mathcal X}_\eps^t 
  \begin{pmatrix}
    \phi_0^\eps\\
a_0^\eps
  \end{pmatrix}
=
\begin{pmatrix}
  \phi^\eps(t)\\
a^\eps(t)
\end{pmatrix},
\end{equation*}
where $(\phi^\eps,a^\eps)$ is the solution to
\eqref{eq:Xbkw}. Similary, we write $\mathcal Z_\eps^t = \mathcal
Y_\eps^t \mathcal X_\eps^t$. 
\subsection{Rewriting the splitting method in the WKB regime}

Instead of analyzing directly the equations
\eqref{eq:linear}--\eqref{eq:Y}, we shall work on
\eqref{eq:Xbkw}--\eqref{eq:Ysplit}, in view of the previous
subsection. We denote by $\Pi^\eps$ the wave reconstruction operator
\begin{equation*}
  \Pi^\eps 
  \begin{pmatrix}
    \phi^\eps\\
a^\eps
  \end{pmatrix}
=
a^\eps e^{i\phi^\eps/\eps},
\end{equation*}
and we note the identity, which is the key conclusion of the above
analysis:
\begin{equation}\label{eq:num}
  \Pi^\eps \mathcal Z_t^\eps 
  \begin{pmatrix}
    \phi^\eps\\
a^\eps
  \end{pmatrix}
=
Z_t^\eps\(a^\eps e^{i\phi^\eps/\eps}\).
\end{equation}
In view of the obvious remark
\begin{equation*}
 \Pi^\eps 
  \begin{pmatrix}
    \phi\\
a
  \end{pmatrix}  =
 \Pi^\eps 
  \begin{pmatrix}
    \phi -\eps \theta\\
a e^{i\theta}
  \end{pmatrix},\quad \forall \theta\in \R, 
\end{equation*}
we see that working with $(\phi^\eps,a^\eps)$ is not equivalent to
working with the wave function $a^\eps e^{i\phi^\eps/\eps}$. However,
we only use the fact that the numerical solution can be represented by
this decomposition, and no uniqueness argument is needed, except the
fact that the solutions to \eqref{eq:linear} and \eqref{eq:Y},
respectively, are unique.  
\smallbreak

We finally notice that the form \eqref{eq:bkw} (with
$a^\eps$ and $\phi^\eps$ bounded in $H^s(\R^d)$
uniformly in $\eps$) is preserved by the exact flow. This is so thanks
to the gauge invariance of the nonlinearity, which
  rules out the appearance of new phases or new harmonics by nonlinear
  interaction (an aspect which also
appears when solving \eqref{eq:Y}). 
In the
case of a local defocusing nonlinearity (typically $f(\rho)=\rho $),
recall the original idea of Grenier \cite{Grenier98} to study the
semi-classical limit for \eqref{eq:nls}: seek the solution $u^\eps$ to
\eqref{eq:nls} under the form \eqref{eq:bkw}, with $\phi^\eps$
real-valued and $a^\eps$ complex-valued. One gains a degree of
freedom, and the choice of Grenier consists in imposing
\begin{equation}
  \label{eq:grenier}
  \left\{
\begin{aligned}
      &\d_t \phi^\eps+\frac{1}{2}|\nabla \phi^\eps|^2 +f\(|a^\eps|^2\)=0 ;\quad
      \phi^\eps_{\mid t=0} =  
\phi_0,\\
&\d_t a^\eps +\nabla \phi^\eps\cdot \nabla a^\eps
+\frac{1}{2}a^\eps\Delta \phi^\eps = i\frac{\eps}{2}\Delta
a^\eps;\quad a^\eps_{\mid t=0}=a_0. 
    \end{aligned}
\right.
\end{equation}
This choice differs from the standard Madelung transform, which is
limited by the presence of vacuum (zeroes of $a^\eps$; see
\cite{CaDaSa12}). 
Also, an important technical feature of
  \eqref{eq:grenier} is that the term $i\Delta a^\eps$ is
  skew-symmetric. Therefore, it plays no role in $H^s$-energy 
  estimates. In particular, it causes no loss of regularity.
Under
Assumption~\ref{hyp:r3}, the adaptation of the  approach of Grenier can be
found in \cite{AC-SP,LiLinEJDE} (see also
\cite{LiuTadmor,Mas10,Mas11b} for the case of 
Schr\"odinger-Poisson system in low dimensions, where low frequencies
cause technical difficulties). We denote by ${\mathcal S}_t^\eps$
the flow associated to \eqref{eq:grenier}. 
\begin{remark}
In \eqref{eq:grenier}, the initial data are supposed implicitly
independent of $\eps$. This is merely for the sake of consistency in
future references. Throughout this paper,
the flow associated to \eqref{eq:grenier} will be considered for
initial data which may depend on $\eps$, but which are uniformly
bounded in suitable Sobolev spaces. 
\end{remark}
\smallbreak

Instead of analyzing directly the splitting
method for \eqref{eq:nls} as presented in Section~\ref{sec:split}, we
shall therefore analyze a splitting method for \eqref{eq:grenier}: when
the term $f$ is discarded, we recover \eqref{eq:Xbkw}, which is solved
alternatingly with \eqref{eq:Ysplit}. 
The latter system consists indeed in dropping out the Laplacian in
\eqref{eq:grenier}, since all spatial derivatives have disappeared.

\section{Technical background}
\label{sec:cadre}
As noticed in \cite{AC-SP}, the following lemma turns out to be
very helpful.
\begin{lemma}\label{lem:gain}
  Let $s\ge 0$. Under Assumption~\ref{hyp:r3}, there
  exists $C$ such that
  \begin{equation}\label{eq:gain}
    \|\nabla f(\rho)\|_{H^{s+1}(\R^d)}\le
    C\(\|\rho\|_{H^s(\R^d)}+\|\rho\|_{L^1(\R^d)}\) ,\quad
    \forall \rho \in H^s(\R^d)\cap L^1(\R^d). 
  \end{equation}
If in addition $s>d/2$, there exists $C$ such that
  \begin{equation}\label{eq:phaseLinfty}
    \|f(\rho)\|_{L^\infty(\R^d)}\le
    C\(\|\rho\|_{H^s(\R^d)}+\|\rho\|_{L^1(\R^d)}\) ,\quad
    \forall \rho \in H^s(\R^d)\cap L^1(\R^d). 
  \end{equation}
\end{lemma}
\begin{proof}
 By Plancherel formula,
 \begin{align*}
   \|\nabla f(\rho)\|_{H^{s+1}(\R^d)}^2 &= \int_{\R^d}
   |\xi|^2\(1+|\xi|^2\)^{s+1} |\widehat 
   K(\xi)|^2|\widehat \rho(\xi)|^2d\xi\\
&\le \(\sup_{\xi\in \R^d} \(1+|\xi|^2\) |\widehat 
   K(\xi)|\)^2 \|\rho\|_{H^s}^2,
 \end{align*}
hence (a weaker version of) the lemma in the first case of
Assumption~\ref{hyp:r3}. 
If $d\ge 3$, 
\begin{align*}
  \int_{|\xi|\le 1}
   |\xi|^2\(1+|\xi|^2\)^{s+1} |\widehat 
   K(\xi)|^2|\widehat \rho(\xi)|^2d\xi &\le \(\sup_{\xi\in \R^d}
   |\xi|^2 |\widehat  
   K(\xi)|\)^2 \int_{|\xi|\le 1}|\xi|^{-2}|\widehat \rho(\xi)|^2d\xi\\
& \le C \|\widehat \rho\|_{L^\infty(\R^d)}^2\int_0^1 r^{d-3}dr\le C
\|\rho\|_{L^1(\R^d)}^2, 
\end{align*}
where we have used spherical coordinates and Hausdorff-Young's
inequality. This yields the first part of the lemma. For the second
part, we use the same tools, 
\begin{equation*}
  \|f(\rho)\|_{L^\infty}\le (2\pi)^{-d/2}\|\widehat{f(\rho)}\|_{L^1} =
  \|\widehat K \widehat \rho\|_{L^1}. 
\end{equation*}
Split the integral between the two regions $\{|\xi|\le 1\}$ and
$\{|\xi|>1\}$:
\begin{align*}
  &\int_{|\xi|\le 1}|\widehat K (\xi)| |\widehat \rho(\xi)|d\xi \le
  C \|\widehat \rho\|_{L^\infty} \int_0^1 r^{d-1}\left\lvert \widehat
    K (r)\right\rvert dr\lesssim
  \|\rho\|_{L^1},\\
&\int_{|\xi|> 1}|\widehat K (\xi)| |\widehat \rho(\xi)|d\xi \le
C\|\widehat \rho\|_{L^1}\lesssim \|\rho\|_{H^s},\quad \text{since }s>d/2.
\end{align*}
This estimate is not sharp, since we do not
use the decay of $\widehat K$ at infinity. 
\end{proof}
As in \cite{AC-SP}, we infer the following result, concerning the
exact solution, that is, the solution to \eqref{eq:grenier}.  This
result implies the first point of Theorem~\ref{theo:main}. 
\begin{proposition}\label{prop:S1}
  Suppose that $d\ge 1$, and that $f$ satisfies
  Assumption~\ref{hyp:r3}. Let $\(\nabla \phi_0,a_0\)\in  
  H^{s+1}\times H^s$ with $s>d/2+1$, and let $T>0$ be such that the solution to
  \eqref{eq:euler} satisfies $(v,\rho)\in C([0,T];H^{s+1}\times
  H^s)$. Then there exists $\eps_0>0$ such that for all $\eps\in
  (0,\eps_0]$, \eqref{eq:grenier} has a unique 
  solution, which satisfies $(\nabla \phi^\eps,a^\eps) \in C([0,T];H^{s+1}\times
  H^s)$,   uniformly in
  $\eps\in (0,\eps_0]$: there exists $C(T,\|a_0\|_{H^s}, \|\nabla
  \phi_0\|_{H^{s+1}})$ independent of $\eps\in (0,\eps_0]$ such that
  \begin{equation*}
    \sup_{t\in [0,T]}\( \|a^\eps(t)\|_{H^s(\R^d)} + \|\nabla
    \phi^\eps(t)\|_{H^{s+1}(\R^d)}\) \le C\(T,\|a_0\|_{H^s(\R^d)}, \|\nabla
  \phi_0\|_{H^{s+1}(\R^d)}\).
  \end{equation*}
If in addition $\phi_0\in L^\infty(\R^d)$, then $\phi^\eps \in
C([0,T];L^\infty (\R^d))$ and
\begin{equation*}
  \sup_{t\in [0,T]}\|
    \phi^\eps(t)\|_{L^\infty(\R^d)} \le \|\phi_0\|_{L^\infty} +
    \underline C\(T,\|a_0\|_{H^s(\R^d)}, \|\nabla 
  \phi_0\|_{H^{s+1}(\R^d)}\).
\end{equation*}
\end{proposition}
\begin{proof}[Sketch of the proof]
 Let $w^\eps=\nabla \phi^\eps$. By differentiating in space the first
equation in \eqref{eq:grenier}, we see that any solution to
\eqref{eq:grenier} must solve
\begin{equation}
  \label{eq:grenierv}
  \left\{
\begin{aligned}
      &\d_t w^\eps+w^\eps\cdot \nabla w^\eps +\nabla f\(|a^\eps|^2\)=0 ;\quad
      w^\eps_{\mid t=0} =  
\nabla \phi_0,\\
&\d_t a^\eps +w^\eps\cdot \nabla a^\eps
+\frac{1}{2}a^\eps\DIV w^\eps = i\frac{\eps}{2}\Delta
a^\eps;\quad a^\eps_{\mid t=0}=a_0. 
    \end{aligned}
\right.
\end{equation}
The left hand side corresponds to a hyperbolic symmetric system for
the unknown $\(w^\eps,\RE a^\eps,\IM a^\eps\)\in 
\R^{d+2}$, thanks to Lemma~\ref{lem:gain}, and the shift in regularity
between $w^\eps\in H^{s+1}$ and $a^\eps\in H^s$. 
The right hand side of \eqref{eq:grenierv} is a skew-symmetric term, which 
does not appear in $H^s$ energy estimates. The key point to notice is
that unlike what would happen in the case of the nonlinear
Schr\"odinger equation,  the terms $\nabla f(|a^\eps|^2)$ and $a^\eps
\DIV w^\eps$ are not 
quasilinear, but semilinear (they can be treated as perturbations), in
view of Lemma~\ref{lem:gain} and the functional framework. By standard
theory (see 
e.g. \cite{AlGe07}), \eqref{eq:grenierv} has  a unique solution
$(w^\eps,a^\eps)\in C([0,\tau];H^{s+1}\times
  H^s)$, for some $\tau>0$ independent of $\eps\in (0,1]$.
\smallbreak

We can take $\tau\ge T$ for $\eps$  sufficiently small.  Indeed, if
$T'$ denotes the lifespan of \eqref{eq:grenierv} in the case $\eps=0$,
then necessarily $T'>T$, for if we had $T'\le T$, then by uniqueness
for the Euler-Poisson system, $|a|^2=\rho\in C([0,T];H^s\cap L^1)$ and $w=v\in
C([0,T];H^{s+1})$. Back to the transport equation in \eqref{eq:grenierv},
we infer that $a\in 
L^\infty([0,T];H^{s})$, which yields a contradiction.  
\smallbreak

Finally, we note that $a^\eps$, $w^\eps$ and $\phi^\eps$ are related through the
formula
\begin{equation*}
  \d_t \phi^\eps+\frac{1}{2}|w^\eps|^2+f\(|a^\eps|^2\)=0;\quad
  \phi^\eps_{\mid t=0}=\phi_0
\end{equation*}
Therefore, if $w^\eps$ and $a^\eps$ are known, then $\phi^\eps$ is obtained by a
simple integration in time, and the last estimate of the proposition
follows from \eqref{eq:phaseLinfty}. 
\end{proof}
Note that the above result is expected to be valid only locally in
time, since the solution to \eqref{eq:euler} may develop a
singularity in finite time. In that case for fixed $\eps>0$, $a^\eps$ may become
singular, or remain smooth but become $\eps$-oscillatory for large time,
as suggested by the simulations in 
\cite{CaMo11}. This can be understood as follows: for large time,
several oscillations are expected in $u^\eps$, so they cannot
be carried by only 
one exponential function as in \eqref{eq:bkw}, therefore, $a^\eps$
becomes 
rapidly oscillatory, and its $H^s$-norm is not bounded uniformly in
$\eps\in(0,1]$. 

The analysis of \cite{AC-SP} also implies the
following result.
\begin{proposition}\label{prop:S2}
  Let $d\ge 1$, and $f$ satisfying Assumption~\ref{hyp:f}. Let $R>0$
  and $s>d/2+2$. There exists $T=T(R)>0$ such that if 
  \begin{equation*}
    \|a_0\|_{H^s(\R^d)}+\|\nabla \phi_0\|_{H^{s+1}(\R^d)}\le R,
  \end{equation*}
then \eqref{eq:euler} has a unique solution $(v,\rho)\in
C([0,T];H^{s+1}\times H^s)$. There exist $\eps_0>0$ and $K=K(R)$
independent of $\eps\in 
(0,\eps_0]$ such that if in addition 
$(\nabla \varphi_0,b_0)\in  
  H^{s+1}\times H^s$ satisfies
 \begin{equation*}
    \|b_0\|_{H^s(\R^d)}+\|\nabla \varphi_0\|_{H^{s+1}(\R^d)}\le R,
  \end{equation*}
then for all $t\in [0,T]$, the solutions to \eqref{eq:grenier} with
initial data $(\phi_0,a_0)$ and $(\varphi_0,b_0)$, respectively, satisfy:
\begin{equation*}
  \|a^\eps(t)-b^\eps(t)\|_{H^s}+\|\nabla \phi^\eps(t)-\nabla
  \varphi^\eps(t)\|_{H^{s+1}}\le K\(\|a_0-b_0\|_{H^s}+\|\nabla \phi_0-\nabla
  \varphi_0\|_{H^{s+1}}\).
\end{equation*}
There exists $\kappa=\kappa(R)$ such that if in addition
$\phi_0,\varphi_0\in L^\infty(\R^d)$,  then 
\begin{equation*}
  \|\phi^\eps(t)-\varphi^\eps(t)\|_{L^\infty}\le
  \|\phi_0-\varphi_0\|_{L^\infty} +\kappa
  \(\|a_0-b_0\|_{H^s}+\|\nabla \phi_0-\nabla 
  \varphi_0\|_{H^{s+1}}\). 
\end{equation*}
\end{proposition}

\section{Estimating the approximate flow}
\label{sec:estimates}

In this section, we prove various estimates concerning the flows
involved in the definition of the numerical scheme, ${\mathcal X}_\eps^t$ and
${\mathcal Y}_\eps^t$. 

\subsection{The nonlinear operator}
\label{sec:nl}

Unlike what happens in most cases when studying splitting operators,
the most delicate operator to control is the linear one, denoted here
by ${\mathcal X}_\eps^t$, while in the present framework, ${\mathcal
  Y}_\eps^t$ turns out to be the 
simpler of the two. 
\begin{lemma}\label{lem:est-Y}
  Let $s>d/2$ and $\phi_0^\eps,a_0^\eps\in \Sch'(\R^d)$, with $(\nabla
  \phi_0^\eps,a_0^\eps)\in H^{s+1}\times H^s$, for some $s>d/2$.  The solution to
  \eqref{eq:Ysplit} is given by
  \begin{equation*}
    \phi^\eps(t)=\phi_0^\eps -tf\(|a_0^\eps|^2\);\quad a^\eps(t)=a_0^\eps.
  \end{equation*}
In particular, there exists
$C=C(\mu)$ such that if $\|a_0^\eps\|_{L^\infty}\le \mu$, 
\begin{equation*}
\|a^\eps(t)\|_{H^s}=\|a_0^\eps\|_{H^s},\quad  \|\nabla \phi^\eps(t)\|_{H^{s+1}} \le
  \|\nabla \phi_0^\eps\|_{H^{s+1}} +Ct \|a_0^\eps\|_{H^s},\quad \forall
  t\ge 0.
\end{equation*}
Finally, if $\phi_0^\eps\in L^\infty(\R^d)$, then there exists
$C=C(\mu)$ such that if $\|a_0^\eps\|_{L^\infty}\le \mu$, 
\begin{equation*}
\| \phi^\eps(t)\|_{L^\infty } \le
  \|\phi_0^\eps\|_{L^\infty}+ Ct \|a_0^\eps\|_{H^s},\quad \forall 
  t\ge 0.
\end{equation*}
\end{lemma}
\begin{proof}
  Since $s>d/2$, $H^s(\R^d)$ is a Banach algebra embedded into
  $C(\R^d)$, hence the formula for $\phi^\eps$. The estimates are 
  straightforward consequences of Lemma~\ref{lem:gain}, and of the
  tame estimate $\|fg\|_{H^s}\lesssim \|f\|_{L^\infty}\|g\|_{H^s}+
  \|f\|_{H^s}\|g\|_{L^\infty}$. 
\end{proof}
\subsection{The linear operator}
\label{sec:lin}

We now consider \eqref{eq:Xbkw}. The following lemma is a variant of
\cite[Lemma~3.2]{HLR13}. Like in that paper, the key
  aspect of the result is the at-most-geometric-growth of $\bv$, which will be
  crucial in the context of the splitting approach, where this control
  will be used on a time step.
\begin{lemma}\label{lem:burgers}
  Let $s>d/2+1$ and $\mu>0$. There exists $\tau=\tau(\mu)>0$ such that if
  \begin{equation*}
    \|\bv_0\|_{H^s}\le \mu,
  \end{equation*}
then the (multi-dimensional)
  Burgers equation
  \begin{equation}
    \label{eq:burgers}
    \d_t \bv+\bv\cdot \nabla \bv=0;\quad \bv_{\mid t=0}=\bv_0
  \end{equation}
has a unique solution $\bv\in C([0,\tau];H^s)$, which satisfies
\begin{equation*}
  \sup_{t\in [0,\tau]}\|\bv(t)\|_{H^s}\le 2\mu. 
\end{equation*}
\end{lemma}
\begin{proof}
    Local existence of a unique $H^s$ solution follows from a global
    inversion theorem (see e.g. \cite{CaBook}), so we focus on the
    energy estimate. We have
  \begin{equation*}
    \frac{1}{2}\frac{d}{dt}\|\bv\|_{H^{s}}^2 = \< \bv,\d_t
    \bv\>_{H^{s}}=\< \Lambda^{s} \bv , \Lambda^{s}\d_t
    \bv\>_{L^2}=- \< \Lambda^{s} \bv , \Lambda^{s}\(\bv\cdot \nabla \bv\)
    \>_{L^2},
  \end{equation*}
where $\Lambda =(1-\Delta)^{1/2}$. Introduce the commutator
\begin{equation*}
  \frac{1}{2}\frac{d}{dt}\|\bv\|_{H^{s}}^2 = - \< \Lambda^{s} \bv ,
  \bv\cdot \nabla \Lambda^{s} \bv
    \>_{L^2} + \< \Lambda^{s} \bv ,\bv\cdot \nabla \Lambda^{s}
    \bv-\Lambda^{s}\(\bv\cdot \nabla \bv\)\>_{L^2}. 
\end{equation*}
By integration by parts, the first term is controlled by
\begin{equation*}
 \left| \< \Lambda^{s} \bv ,
  \bv\cdot \nabla \Lambda^{s} \bv
    \>_{L^2}\right|\le \frac{1}{2}\|\bv\|_{H^{s}}^2 \|\DIV
  \bv\|_{L^\infty} \lesssim \|\bv\|_{H^{s}}^3 ,
\end{equation*}
where we have used Sobolev embedding and the assumption $s>d/2+1$. The
last term is estimated thanks to Kato-Ponce estimate \cite{KaPo88}
\begin{equation}\label{eq:KatoPonce}
  \| \Lambda^{s}(fg)- f\Lambda^{s}g\|_{L^2}\lesssim \|\nabla
  f\|_{L^\infty} \|g\|_{H^{s-1}}+ \|f\|_{H^{s}}\|g\|_{L^\infty},
\end{equation}
with $f=\bv$ and $g =\nabla \bv$:
\begin{align*}
  \left|\< \Lambda^{s} \bv ,\bv\cdot \nabla \Lambda^{s}
    \bv-\Lambda^{s}\(\bv\cdot \nabla \bv\)\>_{L^2}\right| &\le
  \|\bv\|_{H^{s}}\left\| \Lambda^{s}\(\bv\cdot \nabla \bv\)- \bv\cdot
    \nabla \Lambda^{s}\bv \right\|_{L^2} \\
& \lesssim \|\nabla
  \bv\|_{L^\infty}\|\bv\|_{H^{s}}^2\lesssim  \|\bv\|_{H^s}^3 .
\end{align*}
We infer $\frac{d}{dt}\|\bv\|_{H^s}\le C \|\bv\|_{H^s}^2$, and the result
follows by comparing with the ordinary differential equation $\dot
y=Cy^2$. 
\end{proof}
\begin{lemma}\label{lem:condstabburgers}
  Let $s>d/2+1$ and $\mu>0$. If the solution $\bv$ to
  \eqref{eq:burgers} satisfies
  \begin{equation*}
    \|\nabla\bv (t)\|_{L^\infty}\le \mu,\quad 0\le t\le \tau,
  \end{equation*}
 then there exists $c$ independent of $\mu$ and $\tau$ such that
\begin{equation*}
  \sup_{t\in [0,\tau]}\|\bv(t)\|_{H^s}\le e^{c\mu
    t}\|\bv(0)\|_{H^s},\quad 0\le t\le \tau. 
\end{equation*}
\end{lemma}
\begin{proof}
  This lemma is a straightforward consequence of the tame estimates
  used in the proof of Lemma~\ref{lem:burgers}. 
\end{proof}
\begin{proposition}\label{prop:estimXsplit}
  Let $\sigma>d/2$, $\mu>0$. Suppose that
  $(\nabla \phi_0^\eps,a_0^\eps)\in H^{\sigma+1}\times 
  H^\sigma $, with
  \begin{equation*}
    \|\nabla \phi_0^\eps\|_{H^{\sigma+1}}\le \mu,\quad
    \|a_0^\eps\|_{H^\sigma}\le \mu.  
  \end{equation*}
There exists $\tau=\tau(\mu)$ independent of $\eps$ such that
\eqref{eq:Xbkw} has a unique 
solution, with $(\nabla \phi^\eps,a^\eps)\in
C([0,\tau];H^{\sigma+1}\times H^\sigma)$, and 
\begin{equation*}
  \sup_{t\in [0,\tau]}\|\nabla \phi^\eps(t)\|_{H^{\sigma+1}}\le 2\mu,\quad 
\sup_{t\in [0,\tau]}\|a^\eps(t)\|_{H^\sigma}\le 2\mu. 
\end{equation*}
If in addition $\phi_0^\eps\in L^\infty(\R^d)$, then $\phi^\eps \in
C([0,\tau];L^\infty)$ and 
\begin{equation*}
  \sup_{t\in [0,\tau]}\|\phi^\eps(t)\|_{L^\infty}\le
  \|\phi_0^\eps\|_{L^\infty} + \tau \mu. 
\end{equation*}
\end{proposition}
\begin{proof}
  From
  Lemma~\ref{lem:burgers}, \eqref{eq:burgers} has a unique solution
  $\bv\in C([0,\tau];H^{\sigma+1})$, such that $\bv_{\mid t=0}= \nabla
  \phi_0^\eps$, with $\|\bv(t)\|_{H^{\sigma+1}}\le 2\mu$ for $t\in
  [0,\tau]$. Now let 
  \begin{equation*}
    \phi^\eps(t) = \phi_0^\eps -\frac{1}{2}\int_0^t
    |\bv(\sigma)|^2d\sigma. 
  \end{equation*}
We  note that $\d_t(\bv-\nabla \phi^\eps)=\d_t \bv-\nabla \d_t
\phi^\eps=0$, so $\bv=\nabla \phi^\eps$, and the result concerning
$\phi^\eps$ follows. 
\smallbreak

The existence of a solution $a^\eps$ follows for instance from the
fact that it is given by $a^\eps = v^\eps e^{-i\phi^\eps/\eps}$, where
$v^\eps\in C(\R;H^\sigma)$ is the solution to the linear Schr\"odinger equation
\eqref{eq:linear} with initial datum $a_0^\eps
e^{i\phi_0^\eps/\eps}\in H^\sigma$. So we are left with the energy estimate: since
$i\Delta$ is skew-symmetric,
\begin{align*}
  \frac{1}{2}\frac{d}{dt}\|a^\eps\|_{H^{\sigma}}^2 &=
  \<\Lambda^{\sigma}a^\eps, \(\d_t
  -i\frac{\eps}{2}\Delta\)\Lambda^{\sigma}a^\eps
  \>_{L^2}\\
&=-\<\Lambda^{\sigma}a^\eps, \Lambda^\sigma\(\nabla \phi^\eps\cdot \nabla
  a^\eps+\frac{1}{2}a^\eps\Delta \phi^\eps\)\>_{L^2}. 
\end{align*}
By integration by parts,
\begin{equation*}
  \<\Lambda^{\sigma}a^\eps, \nabla \phi^\eps\cdot \nabla
  \Lambda^\sigma a^\eps+\frac{1}{2}\Lambda^\sigma a^\eps\Delta \phi^\eps\>_{L^2}=0,
\end{equation*}
so we have
\begin{align*}
  \frac{1}{2}\frac{d}{dt}\|a^\eps\|_{H^{\sigma}}^2 &=\<\Lambda^{\sigma}a^\eps,
 \nabla\phi^\eps \cdot \nabla \Lambda^\sigma a^\eps- \Lambda^\sigma\(\nabla
 \phi^\eps\cdot \nabla  
  a^\eps\)\>_{L^2}\\
&\quad +\frac{1}{2}\<\Lambda^{\sigma}a^\eps,\Lambda^\sigma a^\eps\Delta
\phi^\eps-\Lambda^\sigma\(a^\eps\Delta \phi^\eps\)\>_{L^2} .
\end{align*}
Kato-Ponce estimate \eqref{eq:KatoPonce} for the first line, and tame
estimates for the second line then yield
\begin{equation}\label{eq:tamea}
  \begin{aligned}
    \frac{d}{dt}\|a^\eps\|_{H^{\sigma}}^2\lesssim
  \|a^\eps\|_{H^{\sigma}}&\Big(\|a^\eps\|_{H^{\sigma}}\|\nabla^2\phi^\eps\|_{L^\infty}
  + \|\nabla \phi^\eps\|_{H^\si}\|\nabla a\|_{L^\infty}\\
&\quad
  +\|\Delta\phi^\eps\|_{L^\infty} \|a^\eps\|_{H^{\si}}+\|\Delta
  \phi^\eps\|_{H^{\si}}\| a\|_{L^\infty}  \Big)\\
  \lesssim \|\nabla
  \phi^\eps&\|_{H^{\si+1}}\|a^\eps\|_{H^{\sigma}}^2,
  \end{aligned}
\end{equation}
since $\si>d/2$,
and the result follows from Gronwall lemma, by decreasing $\tau$
if necessary.
\end{proof}
\begin{remark}\label{rem:shift}
  The above proof suggests that the shift in regularity, between
  $\phi^\eps$ and $a^\eps$, cannot be avoided. Note that this
  phenomenon shows up when the \emph{free} Schr\"odinger equation
  \eqref{eq:linear} is solved, in terms of WKB states,
  and is not a difficulty due to the nonlinear aspect
    of \eqref{eq:nls}.
\end{remark}
\begin{proposition}\label{prop:condstabX}
   Let $\sigma>d/2$, $\mu>0$. Suppose that
 the solution to \eqref{eq:Xbkw} satisfies
  \begin{equation*}
    \|\nabla \phi^\eps(t)\|_{W^{1,\infty}}\le \mu,\quad
    \|a^\eps(t)\|_{W^{1,\infty}}\le \mu,\quad  
   0\le t\le \tau.
  \end{equation*}
There exists $c$ independent of $\eps$, $\mu$, $\tau$ such that
the solution to \eqref{eq:Xbkw} satisfies
\begin{equation*}
  \|\nabla \phi^\eps(t)\|_{H^{\sigma+1}}+ \|a^\eps(t)\|_{H^\sigma}\le
  e^{c\mu t}\(\|\nabla \phi_0^\eps\|_{H^{\sigma+1}}+
  \|a_0^\eps\|_{H^\sigma}\),\quad 0\le t\le \tau. 
\end{equation*}
\end{proposition}
\begin{proof}
  Lemma~\ref{lem:condstabburgers} readily implies
\begin{equation*}
  \|\nabla \phi^\eps(t)\|_{H^{\sigma+1}}\le
  e^{c\mu t}\|\nabla \phi_0^\eps\|_{H^{\sigma+1}},\quad 0\le t\le \tau,
\end{equation*}
for some $c$ independent of $\eps$, $\mu$, $\tau$. Back to the proof
of Proposition~\ref{prop:estimXsplit}, simply apply Gronwall lemma to
\eqref{eq:tamea}. 
\end{proof}
\subsection{The splitting operator}
\label{sec:est-split}

In view of Lemma~\ref{lem:est-Y} and
Proposition~\ref{prop:estimXsplit}, we readily have:
\begin{corollary}\label{cor:est-Z}
  Let $s>d/2$, $\mu>0$. Suppose that $(\nabla \phi_0^\eps,a^\eps)\in
  H^{s+1}\times H^s$, with
  \begin{equation*}
    \|\nabla \phi_0^\eps\|_{H^{s+1}}\le \mu,\quad \|a_0^\eps\|_{H^s}\le \mu. 
  \end{equation*}
There exists $\tau=\tau(\mu)>0$ independent of $\eps$ such that
${\mathcal Z}_\eps^t
\begin{pmatrix}
  \phi_0^\eps \\
a_0^\eps
\end{pmatrix} =
\begin{pmatrix}
  \phi_t^\eps\\
a_t^\eps
\end{pmatrix}$, with 
\begin{equation*}
  \sup_{t\in [0,\tau]}\|\nabla \phi^\eps_t\|_{H^{s+1}}\le 4\mu,\quad 
\sup_{t\in [0,\tau]}\|a^\eps_t\|_{H^{s}}\le 4\mu. 
\end{equation*}
If in addition $\phi_0^\eps\in L^\infty$, with
$\|\phi_0^\eps\|_{L^\infty}\le \mu$, then, up to decreasing $\tau$, we
have
\begin{equation*}
  \sup_{t\in [0,\tau]}\|\phi^\eps_t\|_{L^\infty}\le 4\mu. 
\end{equation*}
\end{corollary}
\begin{proof}
 Lemma~\ref{lem:est-Y}  implies that ${\mathcal Y}_\eps^t\begin{pmatrix}
  \phi_0^\eps \\
a_0^\eps
\end{pmatrix} =
\begin{pmatrix}
  \varphi^\eps_t\\
\alpha_t^\eps
\end{pmatrix}$, with
 \begin{equation*}
   \|\alpha^\eps_t\|_{H^{s}}=\|a_0^\eps\|_{H^s},\quad
   \|\nabla \varphi_t^\eps\|_{H^{s+1}} \le \mu + Ct, \quad t\ge 0. 
 \end{equation*}
We then apply Proposition~\ref{prop:estimXsplit} with $\sigma=
s$. We note that the $L^\infty$ regularity for the phase is propagated by both
operators, and the estimate follows easily.  
\end{proof}
In view of Lemma~\ref{lem:est-Y} and Proposition~\ref{prop:condstabX},
we also infer:
\begin{corollary}\label{cor:condstabZ}
   Let $s>d/2$, $\mu>0$. Suppose that 
${\mathcal Z}_\eps^t
\begin{pmatrix}
  \phi_0^\eps \\
a_0^\eps
\end{pmatrix} =
\begin{pmatrix}
  \phi_t^\eps\\
a_t^\eps
\end{pmatrix}$, with 
\begin{equation*}
  \|\nabla \phi^\eps_t\|_{W^{1,\infty}}\le \mu,\quad 
\|a^\eps_t\|_{W^{1,\infty}}\le \mu,\quad 0\le t\le \tau. 
\end{equation*}
Then there exists $c$ independent of $\eps$, $\mu$, $\tau$, such that
\begin{equation*}
  \|\nabla \phi^\eps_t\|_{H^{s+1}}+
\|a^\eps_t\|_{H^s}\le e^{c\mu t}\(\|\nabla \phi^\eps_0\|_{H^{s+1}}+
\|a^\eps_0\|_{H^s}\),\quad 0\le t\le \tau. 
\end{equation*}
\end{corollary}

\section{Local error estimate}
\label{sec:local}

We recall the result (and resume the notations) from \cite{DeTh13}
concerning the local error 
estimate in the context of \eqref{eq:nls}. For a
possibly nonlinear operator $A$, we denote by $\E_A$ the associated flow:
\begin{equation*}
  \d_t \E_A(t,v)=A\(\E_A(t,v)\);\quad \E_A(0,v)=v.
\end{equation*}
The results presented in this section rely heavily on the following
result. 
\begin{theorem}[Theorem~1 from \cite{DeTh13}]\label{theo:error}
 Suppose that $F(u)=A(u)+B(u)$, and denote by
 \begin{equation*}
   {\mathcal S}^t(u)= \E_F\(t,u\)\text{ and }{\mathcal Z}^t(u)
   =\E_B\(t,\E_A(t,u)\)
 \end{equation*}
the exact flow and the Lie-Trotter flow, respectively. 
 Let $\L (t,u) = {\mathcal Z}^t(u)-{\mathcal S}^t(u)$. We have the exact formula
  \begin{align*}
    \L(t,u) =\int_0^t
    \int_0^{\tau_1}\d_2\E_F& \(t-\tau_1,{\mathcal Z}^{\tau_1}(u)\) \d_2
    \E_B\(\tau_1-\tau_2,\E_A(\tau_1,u)\) \\
&\times [B,A]\(\E_B\(\tau_2,\E_A\(\tau_1,u\)\)\)d\tau_2d\tau_1.
  \end{align*}
\end{theorem}
We emphasize the fact that in \cite{DeTh13}, this result is
established for general operators $A$ and $B$. In particular, both
operators may be nonlinear. 
In the case of \eqref{eq:linear}--\eqref{eq:Y},
\begin{equation*}
  A = i\frac{\eps}{2}\Delta;\quad B(v) = -\frac{i}{\eps}f\(|v|^2\) v;\quad
  F(v)= A(v)+B(v).
\end{equation*}
We have omitted the dependence upon $\eps$ in the notations for the
sake of brevity.
\smallbreak

The linearized flow $\d_2\E_F$ is characterized by $\d_2
\E_F(t,u)w_0=w$, where
\begin{equation*}
  i\eps\d_t w+\frac{\eps^2}{2}\Delta w = f\(|u|^2\) w+f\( \overline u
  w + u\overline w\)u
  ;\quad w_{\mid t=0} =w_0.
\end{equation*}
We note that it is not compatible with our approach, inasmuch as it
\emph{does not preserve the (monokinetic) WKB structure}: if $u=a e^{i\phi/\eps}$
and $w_0 = b_0  e^{i\varphi_0/\eps}$, then the equation becomes
\begin{equation*}
  i\eps\d_t w+\frac{\eps^2}{2}\Delta w = f\(|a|^2\)
  w+f\( \overline a
   e^{-i\phi/\eps}w + ae^{i\phi/\eps}\overline w
  \)a e^{i\phi/\eps}
  ;\quad w_{\mid t=0} =b_0  e^{i\varphi_0/\eps}.
\end{equation*}
In general, this is not compatible with a solution of the form
\begin{equation*}
  w= b^\eps e^{i\varphi^\eps/\eps},
\end{equation*}
with $b^\eps$ and $\varphi^\eps$ uniformly bounded in Sobolev
spaces. Possibly, $w$ should rather be seeked as a superposition of
WKB states,
\begin{equation*}
  w = \sum_j b^\eps_j e^{i\varphi^\eps_j/\eps}.
\end{equation*}
Another, less technical, way to see that the
local error should not be expected to be a single WKB state consists
in going back to the definition. We have seen that the numerical
solution remains of the form (at time $t_n=n\Dt$) $u_n^\eps(x)=a_n^\eps(x)
e^{i\phi_n^\eps(x)/\eps}$, while the exact solution is of the form
(Proposition~\ref{prop:S1}) $u^\eps(t,x)=a^\eps(t,x)
e^{i\phi^\eps(t,x)/\eps}$. Thus the local error is
\begin{equation*}
  \L(t_n,u_0)(x) = a_n^\eps(x)
e^{i\phi_n^\eps(x)/\eps}-a^\eps(t,x)
e^{i\phi^\eps(t,x)/\eps},
\end{equation*}
and it is very unlikely that this can be factored out as
\begin{equation*}
 \L(t_n,u_0)(x)= \alpha_n^\eps(x) e^{i\varphi_n^\eps(x)/\eps}, 
\end{equation*}
with $\alpha_n^\eps$ and $\varphi_n^\eps$ uniformly bounded in Sobolev
spaces (consider for instance the trivial example, 
$\L=\(e^{ix_1/\eps}-1\)e^{-|x|^2}$). 
\smallbreak

This aspect is another motivation for working
with the system \eqref{eq:Xbkw}--\eqref{eq:Ysplit} instead of the standard one 
\eqref{eq:linear}--\eqref{eq:Y}.
We therefore consider the operators $A$ and $B$ defined by
\begin{equation}\label{eq:AB}
  A
  \begin{pmatrix}
    \phi\\
a
  \end{pmatrix}
=
\begin{pmatrix}
  -\frac{1}{2}|\nabla \phi|^2\\
-\nabla \phi\cdot \nabla a -\frac{1}{2}a\Delta \phi
+i\frac{\eps}{2}\Delta a
\end{pmatrix},
\quad B  \begin{pmatrix}
    \phi\\
a
  \end{pmatrix}
=
\begin{pmatrix}
  -f\(|a|^2\)\\
0
\end{pmatrix}.
\end{equation}
We note that with this approach, neither
$A$ nor $B$ is a linear operator. 
\begin{lemma}\label{lem:crochet}
  Let $A$ and $B$ defined by \eqref{eq:AB}. Their commutator is given by
\begin{equation*}
  [A,B]\begin{pmatrix}
    \phi\\
a
  \end{pmatrix} =
  \begin{pmatrix}
    \nabla \phi\cdot \nabla f\(|a|^2\) -\DIV f\(|a|^2\nabla \phi\)
    -\eps \DIV f\(\IM\(\overline a \nabla a\)\)\\
\nabla a\cdot \nabla f\(|a|^2\) +\frac{1}{2}a\Delta f\(|a|^2\)
  \end{pmatrix}.
\end{equation*}
As a consequence, if $s>d/2$, 
$\|\nabla \phi\|_{H^{s+2}}\le M$, $\|a\|_{H^{s+1}}\le M$, then there exists
$C=C(M)$ independent of $\eps\in (0,1]$
such that
\begin{equation*}
  [A,B]\begin{pmatrix}
    \phi\\
a
  \end{pmatrix} =
  \begin{pmatrix}
    \varphi\\
b
  \end{pmatrix}, \quad \text{with }
\left\{
  \begin{aligned}
    \|\varphi\|_{H^{s+2}}&\le C
  \(\|\nabla \phi\|_{H^{s+2}} + \|a\|_{H^{s+1}}\),\\
 \|b\|_{H^{s}}&\le C
  \|a\|_{H^{s+1}}.
  \end{aligned}
\right.
\end{equation*}
In particular,
\begin{equation*}
  \|\varphi\|_{L^\infty} \le C \(\|\nabla \phi\|_{H^{s+2}} +
  \|a\|_{H^{s+1}}\). 
\end{equation*}
\end{lemma}
\begin{proof}
  By definition (see \cite[Section~3]{DeTh13}), 
  \begin{equation*}
    [A,B] v= A'(v)B(v)- B'(v)A(v).
  \end{equation*}
We have, since $f$ is linear in its argument,
\begin{align*}
&  A' \begin{pmatrix}
    \phi\\
a
  \end{pmatrix}
\begin{pmatrix}
    \varphi\\
b
  \end{pmatrix}
=
\begin{pmatrix}
   -\nabla \phi\cdot \nabla \varphi\\
-\nabla \phi\cdot \nabla b -\nabla \varphi\cdot \nabla a
-\frac{1}{2}b\Delta \phi-\frac{1}{2}a\Delta
\varphi+i\frac{\eps}{2}\Delta b
  \end{pmatrix},\\
& B' \begin{pmatrix}
    \phi\\
a
  \end{pmatrix}
\begin{pmatrix}
    \varphi\\
b
  \end{pmatrix}
=
\begin{pmatrix}
  -f\(\overline a b+a\overline b\)\\
0
\end{pmatrix}=
\begin{pmatrix}
  -2f\(\RE \(\overline a b\)\)\\
0
\end{pmatrix}.
\end{align*}
We compute
\begin{equation*}
  B' \begin{pmatrix}
    \phi\\
a
  \end{pmatrix}
A \begin{pmatrix}
    \phi\\
a 
  \end{pmatrix}
=
\begin{pmatrix}
 2f\(\RE\(\overline a \nabla \phi\cdot \nabla a \)\)+f \(|a|^2\Delta
 \phi\)+ \eps f\(\IM\(\overline a \Delta a\)\) \\
0
\end{pmatrix}
\end{equation*}
The main point is then to notice the  factorizations
\begin{equation*}
  2 \RE\(\overline a \nabla \phi\cdot \nabla a \) +|a|^2\Delta
 \phi = \DIV\( |a|^2\nabla \phi\),\quad  \IM\(\overline a \Delta a\) =
 \DIV \IM\(\overline a \nabla a\),
\end{equation*}
and to recall $\d_j f(\rho)=f(\d_j \rho)$, $1\le j\le d$. 

The estimates of the lemma then follow from the explicit formula for
$[A,B]$, from the fact that $H^{s+2}(\R^d)$, $H^{s+1}(\R^d)$ and
$H^s(\R^d)$ are Banach algebras, from \eqref{eq:gain}, and from the
embedding $H^{s+2}\hookrightarrow L^\infty$. 
\end{proof}
We have the explicit formula
\begin{equation*}
 {\mathcal Y}_\eps^t\begin{pmatrix}
    \phi\\
a
  \end{pmatrix}=  \E_B\(t, \begin{pmatrix}
    \phi\\
a
  \end{pmatrix}\) 
=
\begin{pmatrix}
  \phi-tf\(|a|^2\)\\
a
\end{pmatrix},
\end{equation*}
and we readily infer
\begin{equation}\label{eq:d2EB}
   \d_2 \E_B\(t, \begin{pmatrix}
    \phi\\
a
  \end{pmatrix}\) \begin{pmatrix}
    \varphi\\
b
  \end{pmatrix}
=
\begin{pmatrix}
  \varphi-  2t\RE f \(\overline a b\)\\
b
\end{pmatrix}.
\end{equation}
Finally, we compute
\begin{equation*}
  \d_2\E_F\(t,\begin{pmatrix}
    \phi\\
a
  \end{pmatrix}
\)\begin{pmatrix}
    \varphi_0\\
b_0
  \end{pmatrix}
=
\begin{pmatrix}
  \varphi(t)\\
b(t)
\end{pmatrix},\quad \text{where}
\end{equation*}
\begin{equation}\label{eq:linearise}
  \left\{
    \begin{aligned}
      & \d_t \varphi + \nabla \phi\cdot \nabla \varphi +
      2\RE f\(\overline a b\) =0;\quad \varphi_{\mid t=0}=\varphi_0,\\
& \d_t b +\nabla \phi\cdot \nabla b+\nabla \varphi\cdot \nabla a +
\frac{1}{2}\( b\Delta \phi+a\Delta \varphi\)=i\frac{\eps}{2}\Delta
b;\quad b_{\mid t=0}=b_0. 
    \end{aligned}
\right.
\end{equation}
\begin{lemma}\label{lem:flotlinearise}
  Let $s>d/2$. Assume that $(\nabla \phi,a)\in
  L^1(I;H^{s+2}\times H^{s+1})$, where $0\in I$.   There
  exists $C$ independent of $\eps\in (0,1]$ 
  such that if $(\nabla \varphi_0, b_0)\in H^{s+1}\times H^s$, the
  solution to \eqref{eq:linearise} satisfies 
  for all $t\in I$,
  \begin{equation*}
    \|b(t)\|_{H^s}+\|\nabla \varphi(t)\|_{H^{s+1}} \le \(
    \|b_0\|_{H^s}+\|\nabla \varphi_0\|_{H^{s+1}} \) e^{C\int_0^t\(
      \|a(\tau)\|_{H^{s+1}} + 
      \|\nabla \phi(\tau)\|_{H^{s+2}}\)d\tau}.
  \end{equation*}
If in addition
$\varphi_0\in L^\infty$, then
\begin{equation*}
  \|\varphi(t)\|_{L^\infty} \le \( \|\varphi_0\|_{L^\infty}
  +\|b_0\|_{H^s}+\|\nabla \varphi_0\|_{H^{s+1}} \) e^{C\int_0^t\(
      \|a(\tau)\|_{H^{s+1}} + 
      \|\nabla \phi(\tau)\|_{H^{s+2}}\)d\tau}.
\end{equation*}
\end{lemma}
\begin{proof}
Set $w=\nabla \varphi$: \eqref{eq:linearise} implies
\begin{equation}\label{eq:linearisev}
  \left\{
    \begin{aligned}
      & \d_t w + \nabla \phi\cdot \nabla w +\nabla^2 \phi\cdot w
      +2\RE\nabla f \(\overline a b\) =0;\quad
      w_{\mid t=0}=\nabla\varphi_0,\\ 
& \d_t b +\nabla \phi\cdot \nabla b+w\cdot \nabla a +
\frac{1}{2}\( b\Delta \phi+a\DIV w\)=i\frac{\eps}{2}\Delta
b;\quad b_{\mid t=0}=b_0. 
    \end{aligned}
\right.
\end{equation}
As in the proof of Proposition~\ref{prop:S1}, the term $i\Delta
b$ being skew-symmetric, it does not show up in energy
estimates. Using Lemma~\ref{lem:gain}, we have the estimate
\begin{align*}
  \|w(t)\|_{H^{s+1}} + \|b(t)\|_{H^s} &\le \|w_0\|_{H^{s+1}} +
  \|b_0\|_{H^s}\\
 +C\int_0^t &\( \|\nabla \phi(\tau)\|_{H^{s+2}}
  +\|a(\tau)\|_{H^{s+1}}\) \( \|w(\tau)\|_{H^{s+1}} + \|b(\tau)\|_{H^s} \)d\tau,
\end{align*}
and the first estimate of the lemma stems from Gronwall lemma. 

The second estimate then follows from the first equation in
\eqref{eq:linearise} (integrated in time), and
\eqref{eq:phaseLinfty}. 
\end{proof}
Putting these estimates together, and using Theorem~\ref{theo:error},
we obtain a result which is crucial in the proof of
Theorem~\ref{theo:main}: 
\begin{theorem}[Local error estimate for WKB
  states]\label{theo:errlocBKW} 
  Let $s>d/2+1$ and $\mu>0$. Suppose that
  \begin{equation*}
    \|\nabla \phi^\eps\|_{H^{s+1}}\le \mu,\quad \|a^\eps\|_{H^s}\le \mu.
  \end{equation*}
There exist $C,c_0>0$ (depending on $\mu$) independent
  of $\eps\in (0,1]$ such that 
  \begin{equation*}
    \L\(t,
    \begin{pmatrix}
      \phi^\eps\\
a^\eps
    \end{pmatrix}\):= \mathcal Z^t_\eps\begin{pmatrix}
      \phi^\eps\\
a^\eps
    \end{pmatrix}-
\mathcal S^t_\eps\begin{pmatrix}
      \phi^\eps\\
a^\eps
    \end{pmatrix}  =
    \begin{pmatrix}
      \Psi^\eps(t)\\
A^\eps(t)
    \end{pmatrix},
  \end{equation*}
where $A^\eps$ and $\Psi^\eps$ satisfy 
\begin{equation*}
  \|\nabla \Psi^\eps(t)\|_{H^{s}}+\|A^\eps(t)\|_{H^{s-1}}\le C
  t^2,\quad 0\le t\le c_0. 
\end{equation*}
If in addition $\|\phi^\eps\|_{L^\infty}\le \mu$, then (up to
increasing $C$)
\begin{equation*}
 \| \Psi^\eps(t)\|_{L^\infty}\le C  t^2,\quad 0\le t\le c_0. 
\end{equation*}
\end{theorem}
\begin{proof}
  Let $t\in [0,c]$, and fix $\tau_1,\tau_2$ such that $0\le \tau_2\le
  \tau_1\le t$. Introduce the following intermediary notations:
  \begin{align*}
    &\E_A\(\tau_1,
    \begin{pmatrix}
      \phi^\eps\\
a^\eps
    \end{pmatrix}\)= 
    \begin{pmatrix}
      \phi_1^\eps\\
a_1^\eps
    \end{pmatrix},\\
&\E_B\(\tau_2,
\begin{pmatrix}
      \phi_1^\eps\\
a_1^\eps
    \end{pmatrix}\) = 
   \begin{pmatrix}
      \phi_2^\eps\\
a_2^\eps
    \end{pmatrix},\\
&[B,A] \begin{pmatrix}
      \phi_2^\eps\\
a_2^\eps
    \end{pmatrix}= 
\begin{pmatrix}
      \phi_3^\eps\\
a_3^\eps
    \end{pmatrix},\\
& \d_2 \E_B\(\tau_1-\tau_2, 
\begin{pmatrix}
  \phi_1^\eps\\
a_1^\eps
\end{pmatrix}\)
\begin{pmatrix}
      \phi_3^\eps\\
a_3^\eps
    \end{pmatrix}
=
\begin{pmatrix}
      \phi_4^\eps\\
a_4^\eps
    \end{pmatrix},\\
&\E_B\(\tau_1,
    \begin{pmatrix}
      \phi^\eps\\
a^\eps
    \end{pmatrix}\) = 
    \begin{pmatrix}
      \tilde \phi_1^\eps\\
\tilde a_1^\eps
    \end{pmatrix},\\
&\E_A\(\tau_1,
\begin{pmatrix}
      \tilde \phi_1^\eps\\
\tilde a_1^\eps
    \end{pmatrix}\)=
\begin{pmatrix}
      \tilde \phi_2^\eps\\
\tilde a_2^\eps
    \end{pmatrix}
  \end{align*}
  Then in view of Theorem~\ref{theo:error}, we have
  \begin{equation*}
    \begin{pmatrix}
      \Psi^\eps\\
A^\eps
    \end{pmatrix}=
\int_0^t\int_0^{\tau_1}\d_2\E_F\(t-\tau_1,
\begin{pmatrix}
      \tilde \phi_2^\eps\\
\tilde a_2^\eps
    \end{pmatrix}\)
\begin{pmatrix}
      \phi_4^\eps\\
a_4^\eps
    \end{pmatrix}d\tau_2d\tau_1.
  \end{equation*}
In view of Proposition~\ref{prop:estimXsplit}, we have, uniformly on
$[0,c]$, for $c$
sufficiently small, 
\begin{equation*}
  \|\nabla \phi_1^\eps\|_{H^{s+1}}\le 2\mu,\quad \|a_1^\eps\|_{H^s}\le 2\mu. 
\end{equation*}
Now Lemma~\ref{lem:est-Y} implies (up to decreasing $c$)
\begin{equation*}
  \|\nabla \phi_2^\eps\|_{H^{s+1}}\le 3\mu,\quad \|a_2^\eps\|_{H^s}\le 3\mu. 
\end{equation*}
From Lemma~\ref{lem:crochet}, we infer
\begin{equation*}
  \|\nabla\phi_3^\eps\|_{H^{s}}\le 4\mu,\quad \|a_3^\eps\|_{H^{s-1}}\le 4\mu,
\end{equation*}
provided that $s-1>d/2$. In view of \eqref{eq:d2EB}, we have
\begin{equation*}
  a_4^\eps=a_3^\eps,\quad\phi_4^\eps = \phi_3^\eps
-2(\tau_1-\tau_2) \RE f\(\overline a_1^\eps a_3^\eps\),
\end{equation*}
and therefore
\begin{equation*}
    \|\nabla \phi_4^\eps\|_{H^{s-1}}\le 5\mu,\quad \|a_4^\eps\|_{H^{s-1}}\le 5\mu,
\end{equation*}
since $s-1>d/2$. Now Corollary~\ref{cor:est-Z} implies
\begin{equation*}
   \|\nabla \tilde \phi_2^\eps\|_{H^{s+1}}\le 4\mu,\quad \|\tilde
   a_2^\eps\|_{H^{s}}\le 4\mu.
\end{equation*}
Finally,  Lemma~\ref{lem:flotlinearise} yields, up to
decreasing $c$ one last time,
\begin{equation*}
 \d_2\E_F\(t-\tau_1,
\begin{pmatrix}
      \tilde \phi_2^\eps\\
\tilde a_2^\eps
    \end{pmatrix}\)
\begin{pmatrix}
      \phi_4^\eps\\
a_4^\eps
    \end{pmatrix}=
    \begin{pmatrix}
      \theta^\eps\\
\alpha^\eps
    \end{pmatrix},\text{ with }
    \|\nabla\theta^\eps\|_{H^{s}}\le 10\mu,\quad
    \|\alpha^\eps\|_{H^{s-1}}\le 10\mu. 
\end{equation*}
The first estimate of the theorem then follows by integrating with
respect to $(\tau_1,\tau_2)$ on $\{0\le \tau_2\le  \tau_1\le t\}$. The
$L^\infty$-estimate of the phase follows similarly. 
\end{proof}
Back to the wave functions, we obtain an
estimate similar to the one presented in \cite[Section~4.2.2]{DeTh13}:
\begin{corollary}\label{cor:local-wave}
    Let $s>d/2+1$ and $\mu>0$. Let $\phi_0^\eps\in L^\infty ,a_0^\eps\in H^s$ with
  \begin{equation*}
 \|\phi_0^\eps\|_{L^\infty}\le \mu,\quad
 \|\nabla\phi_0^\eps\|_{H^{s+1}}\le \mu,\quad \|a_0^\eps\|_{H^s}\le
 \mu. 
  \end{equation*}
There exist $C,c_0>0$ (depending on $\mu$) independent
  of $\eps\in (0,1]$ such that 
  \begin{equation*}
    \left\| Z_\eps^t\(a_0^\eps e^{i\phi_t^\eps/\eps}\) -S_\eps^t \(a_0^\eps
      e^{i\phi_0^\eps/\eps}\) \right\|_{L^2}\le
    C\frac{t^2}{\eps}, \quad 0\le t\le c_0.
  \end{equation*}
\end{corollary}
\begin{proof}
  We have 
$S_\eps^t u_0^\eps =a^\eps(t)e^{i\phi^\eps(t)/\eps}$ where $a^\eps$
  and $\phi^\eps$ are given by Proposition~\ref{prop:S1}, and
  \begin{equation*}
    a_t^\eps -a^\eps(t)=A^\eps(t),\quad \phi_t^\eps-\phi^\eps(t)=\Psi^\eps(t),
  \end{equation*}
where
  $A^\eps$ and $\Psi^\eps$ are given by Theorem~\ref{theo:errlocBKW}. We
  compute, since
  $\|a^\eps(t)\|_{L^2}=\|u^\eps(t)\|_{L^2}=\|a_0^\eps\|_{L^2}$, 
  \begin{align*}
    \left\| Z_\eps^t\(a_0^\eps e^{i\phi_t^\eps/\eps}\) -S_\eps^t
      u_0^\eps\right\|_{L^2}&=  
\left\| a_t^\eps(t)e^{i\phi_t^\eps/\eps}
  -a^\eps(t)e^{i\phi^\eps(t)/\eps}\right\|_{L^2} \\
&\le \|a_t^\eps
-a^\eps(t)\|_{L^2} +\left\|a^\eps(t)\(e^{i\phi_t^\eps/\eps}-
e^{i\phi^\eps(t)/\eps}\)\right\|_{L^2} \\
&\le \|A^\eps(t)\|_{L^2} + \|a^\eps(t)\|_{L^2}
\left\|\frac{\phi_t^\eps-\phi^\eps(t)}{2\eps}\right\|_{L^\infty} \\
&\le C t^2 + \frac{\mu}{2\eps}\|\Psi^\eps(t)\|_{L^\infty}\lesssim
\frac{t^2}{\eps},
  \end{align*}
where we have used Theorem~\ref{theo:errlocBKW}.
\end{proof}
\begin{corollary}[Local error for quadratic observables]\label{cor:quad-loc}
  Let $s>d/2+1$ and $\mu>0$. Let $\phi^\eps \in L^\infty,a^\eps\in
  H^s$ with 
  \begin{equation*}
   \|\phi^\eps\|_{L^\infty}\le \mu,\quad
   \|\nabla\phi^\eps\|_{H^{s+1}}\le \mu,\quad \|a^\eps\|_{H^s}\le
   \mu. 
  \end{equation*}
There exist $C,c_0>0$ independent
  of $\eps\in (0,1]$ such that for $0\le t\le c_0$, and $u_0^\eps =
  a_0^\eps e^{i\phi_0^\eps/\eps}$, 
 \begin{align*}
  &\left\| \left\lvert Z^t_\eps u_0^\eps \right\rvert^2 - \left\lvert
      S^t_\eps u_0^\eps \right\rvert^2 \right\|_{L^1(\R^d)\cap
    L^\infty(\R^d)}\le Ct^2,\\ 
&\left\| \IM \(\eps \overline{Z^t_\eps u_0^\eps} \nabla
 Z^t_\eps u_0^\eps  \) - \IM
  \(\eps \overline{S^t_\eps u_0^\eps} \nabla S^t_\eps u_0^\eps\)
\right\|_{L^1(\R^d)\cap    L^\infty(\R^d)}\le Ct^2.
  \end{align*} 
\end{corollary}
\begin{proof}
  Resuming the notations from the proof of
  Corollary~\ref{cor:local-wave}, we have
  \begin{equation*}
    \left\lvert Z^t_\eps u_0^\eps \right\rvert^2 - \left\lvert
      S^t_\eps u_0^\eps \right\rvert^2 = |a_t^\eps|^2-|a^\eps(t)|^2,
  \end{equation*}
and the Cauchy-Schwarz inequality yields
\begin{equation*}
 \left\| |a_t^\eps|^2-|a^\eps(t)|^2\right\|_{L^1}\le
 \|a_t^\eps-a^\eps(t)\|_{L^2}\( \|a_t^\eps\|_{L^2}+\|a^\eps(t)\|_{L^2}\).
\end{equation*}
The first part of the corollary then stems from
Theorem~\ref{theo:errlocBKW}.
Similarly,
\begin{align*}
  \IM \(\eps \overline{Z^t_\eps u_0^\eps} \nabla
  Z^t_\eps u_0^\eps\) - \IM
  \(\eps \overline{S^t_\eps u_0^\eps} \nabla S^t_\eps u_0^\eps\)& =
  |a_t^\eps|^2\nabla \phi_t^\eps-|a^\eps(t)|^2 \nabla
  \phi^\eps(t)\\
&\quad +\eps\IM \(\overline{a^\eps_t}\nabla a^\eps_t\) -
  \eps\IM\(\overline{a^\eps(t)}\nabla a^\eps(t)\). 
\end{align*}
The second part of the corollary then follows easily from
H\"older inequality and Theorem~\ref{theo:errlocBKW}.
\end{proof}

\section{End of the proof of Theorem~\ref{theo:main}}
\label{sec:end}

\subsection{Lady Windermere's fan}
\label{sec:numer}

We denote 
\begin{equation*}
  \begin{pmatrix}
    \phi_n^\eps\\
a_n^\eps
  \end{pmatrix} = \({\mathcal Z}_\eps^\Dt\)^n
 \begin{pmatrix}
    \phi_0\\
a_0
  \end{pmatrix}.
\end{equation*}
To prove Theorem~\ref{theo:main}, we rephrase it in a more
precise way:
\begin{proposition}\label{prop:bound}
  Let $s>d/2+2$, $\phi_0\in L^\infty,a_0\in H^s$, with $\nabla
  \phi_0\in H^{s+1}$, and $T$ as in
  Theorem~\ref{theo:main}. There exist $\nu,\gamma,\Dt_0,c_1,C_0>0$
  such that for all $\eps\in (0,1]$, 
  all $0\le \Dt\le \Dt_0$ and all $n\in \N$ such that  $t_n=n\Dt\in [0,T]$, 
  \begin{align}
\label{eq:s-5}    &
\|\nabla\phi_n^\eps\|_{H^{s}}+ \|a_n^\eps\|_{H^{s-1}}\le 
    \nu,\\
\label{eq:err-5}&\|\nabla\phi^\eps_n
-\nabla\phi^\eps(t_n)\|_{H^{s}}+  
\|a^\eps_n -a^\eps(t_n)\|_{H^{s-1}}\le \gamma \Dt,\\
\label{eq:s}& \|\nabla \phi_n^\eps\|_{H^{s+1}}+ \|a_n^\eps\|_{H^{s}}\le
  e^{c_1 \nu n\Dt}\le  C_0=e^{c_1 \nu T},\\
\label{eq:s+}    &\|\phi^\eps_n
-\phi^\eps(t_n)\|_{L^\infty} \le 
    \gamma \Dt.
  \end{align}
\end{proposition}
\begin{remark}[$L^\infty$ bounds]
  The above result has an important technical consequence:  the
  numerical solution $u_n^\eps= a_n^\eps e^{i\phi_n^\eps/\eps}$ is
  uniformly bounded in $L^\infty(\R^d)$. In view of
  Proposition~\ref{prop:S1}, the same holds for the exact solution
  $u^\eps(t)$. Such informations are very
  delicate to obtain in general. Even in one dimension, the
  Gagliardo-Nirenberg inequality
  \begin{equation*}
    \|u^\eps\|_{L^\infty}\le \sqrt 2\|u^\eps\|_{L^2}^{1/2}\|\d_x u^\eps\|_{L^2}^{1/2}
  \end{equation*}
would not yield better than $\|u^\eps\|_{L^\infty}\lesssim
\eps^{-1/2}$, because of the rapid oscillations present in $u^\eps$
($\phi^\eps\not =0$). 
Here, the uniform $L^\infty$ estimates follow from the fact that
  a WKB regime is considered. 
\end{remark}
\begin{proof}
  The proof that we present follows essentially the lines of
  \cite[Section~5]{HLR13}. 
Denote by $
\begin{pmatrix}
  \phi_k\\
a_k
\end{pmatrix}
 = \({\mathcal Z}_L^{\Dt}\)^k
 \begin{pmatrix}
   \phi_0\\
a_0
 \end{pmatrix}$ the numerical
  solution,  and
  \begin{equation*}
    \begin{pmatrix}
      \phi_n^k\\
a_n^k
    \end{pmatrix}
= {\mathcal S}^{(n-k)\Dt}
 \begin{pmatrix}
      \phi_k\\
a_k
    \end{pmatrix}
.
  \end{equation*}
In this proof, we omit the
dependence of all the functions with respect to $\eps$. 
 From
Proposition~\ref{prop:S1}, there exists $R$ such that 
\begin{equation*}
\left\|
  \begin{pmatrix}
    \phi(t)\\
a(t)
  \end{pmatrix}
\right\|_{\nabla H^{s+1}\times H^s}:=
\|\nabla\phi(t)\|_{H^{s+1}}+\|a(t)\|_{H^s}\le R,\quad \forall t\in [0,T].
\end{equation*}
We prove Proposition~\ref{prop:bound} by induction, with
$\nu=R+\delta$, $\delta>0$ so that the solution to \eqref{eq:euler}
with data in the ball characterized by \eqref{eq:s-5} remains smooth
up to time $T$ (this is possible, since $T<T_{\rm max}$).
The estimates are obviously satisfied for $n=0$. Let $n\ge 1$, and suppose
that the induction assumption is true for $0\le k\le n-1$. 
We introduce the same telescopic series as in \cite{HLR13}, which is
different from \eqref{eq:lady1}, the latter being useful mostly when the
problem (hence the splitting operator) is linear:
\begin{equation}\label{eq:lady2}
  \begin{pmatrix}
    \phi_n^\eps\\
a_n^\eps
  \end{pmatrix}
-
\begin{pmatrix}
  \phi^\eps(t_n)\\
a^\eps(t_n)
\end{pmatrix}=
\sum_{j=0}^{n-1} \(\mathcal S_\eps^{(n-j-1)\Dt}\mathcal Z_\eps^\Dt 
 \begin{pmatrix}
    \phi_j^\eps\\
a_j^\eps
  \end{pmatrix}
  - \mathcal S_\eps^{(n-j-1)\Dt}\mathcal S_\eps^{\Dt}
\begin{pmatrix}
    \phi_j^\eps\\
a_j^\eps
  \end{pmatrix}
\).
\end{equation}
Noting the properties $f_n=f_n^n$ and $f(t_n)= f_n^0$ ($f=\phi$ or
$a$), we estimate 
\begin{align*}
  \|\nabla \phi_{n}-&\nabla\phi(t_{n})\|_{H^{s}}+
  \|a_{n}-a(t_{n})\|_{H^{s-1}}\\
\le
  \sum_{k=0}^{n-1} &
 \(\|\nabla\phi_{n}^{k+1}-\nabla\phi_n^k\|_{H^{s}}+
 \|a_{n}^{k+1}-a_n^k\|_{H^{s-1}}\) 
 \\ 
\le \sum_{k=0}^{n-1}&\left\| {\mathcal S}^{(n-k-1)\Dt}\(\mathcal
  Z^{\Dt}
  \begin{pmatrix}
    \phi_k\\
a_k
  \end{pmatrix}
\)
  - \mathcal S^{(n-k-1)\Dt}\(\mathcal S^{\Dt}
 \begin{pmatrix}
    \phi_k\\
a_k
  \end{pmatrix}
\)\right\|_{\nabla H^{s}\times H^{s-1}}.
\end{align*}
For $k\le n-2$, ${\mathcal Z}_L^{\Dt}
\begin{pmatrix}
  \phi_k\\
a_k
\end{pmatrix}
= 
\begin{pmatrix}
  \phi_{k+1}\\
a_{k+1}
\end{pmatrix}$ and
Proposition~\ref{prop:S2} yields, along with the induction
assumption (all the norms are in $\nabla H^{s}\times H^{s-1}$),
\begin{align*}
\left\| \mathcal S^{\Dt} \begin{pmatrix}
    \phi_k\\
a_k
  \end{pmatrix}\right\| &\le 
\left\| \mathcal S^{\Dt} \begin{pmatrix}
    \phi_k\\
a_k
  \end{pmatrix}
-
\mathcal S^{\Dt} \begin{pmatrix}
    \phi (t_k)\\
a (t_k)
  \end{pmatrix}\right\|
+ \left\| \mathcal S^{\Dt} \begin{pmatrix}
    \phi (t_k)\\
a (t_k)
  \end{pmatrix}\right\| \\
&\le K(2R)\left\| \begin{pmatrix}
    \phi_k\\
a_k
  \end{pmatrix}
-
\begin{pmatrix}
    \phi (t_k)\\
a (t_k)
  \end{pmatrix}\right\|
+\left\|  \begin{pmatrix}
    \phi (t_{k+1})\\
a (t_{k+1})
  \end{pmatrix}\right\|\\
&\le K\gamma \Dt + R,
\end{align*}
which is bounded by $R+\delta$ if $0<\Dt\le \Dt_0\ll 1$. Up to
replacing $K$ with $\max(K,1)$, we obtain that, for $k\le n-1$ and $n\Dt\le T$, 
\begin{equation*}
\left\| \mathcal S^{(n-k-1)\Dt}\(\mathcal Z^{\Dt} 
  \begin{pmatrix}
    \phi_k\\
a_k
  \end{pmatrix}
\) 
-
\mathcal S^{(n-k-1)\Dt}\(\mathcal S^{\Dt} 
  \begin{pmatrix}
    \phi_k\\
a_k
  \end{pmatrix}
\) \right\|_{\nabla H^{s}\times H^{s-1}}
\end{equation*}
is controlled by
\begin{equation*}
K
\left\| \mathcal Z^{\Dt} 
  \begin{pmatrix}
    \phi_k\\
a_k
  \end{pmatrix}
-
\mathcal S^{\Dt} 
  \begin{pmatrix}
    \phi_k\\
a_k
  \end{pmatrix}
\right\|_{\nabla H^{s}\times H^{s-1}}.
\end{equation*}
Using the  local error estimate from Theorem~\ref{theo:errlocBKW}, we
infer, using \eqref{eq:s},
\begin{align*}
\left\| \mathcal S^{(n-k-1)\Dt}\(\mathcal Z^{\Dt} 
  \begin{pmatrix}
    \phi_k\\
a_k
  \end{pmatrix}
\) 
-
\mathcal S^{(n-k-1)\Dt}\(\mathcal S^{\Dt} 
  \begin{pmatrix}
    \phi_k\\
a_k
  \end{pmatrix}
\) \right\|_{H^{s-5}}\le C K\(\Dt\)^2,
\end{align*}
for some uniform constant $C$ depending on $C_0$.  
Therefore,
\begin{align*}
  \|\nabla\phi_{n}-\nabla\phi(t_{n})\|_{H^{s}}+ \|a_{n}-a(t_{n})\|_{H^{s-1}}\le
  CTK\Dt,
\end{align*}
and we can take $\gamma = CTK$,
which is uniform in $n$ and $\Dt$, in order to get \eqref{eq:s-5} and
\eqref{eq:err-5}. Then \eqref{eq:s} follows from
Corollary~\ref{cor:condstabZ}, in view of \eqref{eq:s-5} and Sobolev
embedding, since we have assumed $s>d/2+2$. Finally, the
$L^\infty$-estimates \eqref{eq:s+} for $\phi_n^\eps$ are now
straightforward (up to increasing $\gamma$), and
are left out.  
\end{proof}

\begin{remark}[Nonlinear Schr\"odinger equation]\label{rem:loss}
  We can now explain why Assumption~\ref{hyp:r3} is needed for the
  complete argument to work out. If we wanted to prove the analogue of
  Theorem~\ref{theo:main} for, say, the defocusing cubic Schr\"odinger
  equation 
  \begin{equation*}
    i\eps \d_t u^\eps +\frac{\eps^2}{2}\Delta u^\eps =
 |u^\eps|^2 u^\eps,
  \end{equation*}
then many results would still be available. In terms of the numerical
scheme, the only change would affect the operator ${\mathcal
  Y}^t_\eps$: \eqref{eq:Ysplit} would be replaced by
\begin{equation*}
\left\{ 
 \begin{aligned}
    & \d_t \phi^\eps + |a^\eps|^2=0;\quad \phi^\eps_{\mid
      t=0}=\phi_0^\eps,\\
&\d_t a^\eps=0;\quad a^\eps_{\mid t=0}=a_0^\eps. 
  \end{aligned}
\right.
\end{equation*}
Working in $H^s$ for $s>d/2$, we see that unlike what happens under
Assumption~\ref{hyp:r3}, $\phi^\eps$ cannot be more regular than
$a_0^\eps$. On the other hand, the WKB formulation of the free
Schr\"odinger flow \eqref{eq:Xbkw} induces a shift of regularity: if
$\phi^\eps$ is in $H^s$ for $s$ large, then $a^\eps$ must not be
expected to be more regular than $H^{s-2}$. Therefore, the splitting
operator $\mathcal Z^t_\eps$ induces a loss of regularity, and this
loss is iterated like $T/\Dt$ times. It is this aspect which makes it
hard to adapt Proposition~\ref{prop:bound} to the case of the
nonlinear Schr\"odinger equation. 
\end{remark}

\subsection{Proof of Corollary~\ref{cor:wave}}
\label{sec:cor1}

Once Theorem~\ref{theo:main} is available, we simply write, like in
the proof of Corollary~\ref{cor:local-wave}, 
\begin{align*}
  \(Z_\eps^\Dt\)^n u_0^\eps  -S^{t_n}_\eps   u_0^\eps &= 
a_n^\eps e^{i\phi_n^\eps/\eps}-a^\eps(t_n)e^{i\phi^\eps(t_n)/\eps}\\
&= \(a_n^\eps- a^\eps(t_n)\)e^{i\phi_n^\eps/\eps} +
a^\eps(t_n)\(e^{i\phi_n^\eps/\eps}-  e^{i\phi^\eps(t_n)/\eps}\).
\end{align*}
Taking the $L^2$-norm, we infer
\begin{equation*}
  \left\| \(Z_\eps^\Dt\)^n u_0^\eps   -S^{t_n}_\eps
    u_0^\eps\right\|_{L^2} \le \|a_n^\eps- a^\eps(t_n)\|_{L^2} +
  \|a^\eps(t_n)\|_{L^2} \left\|\frac{\phi_n^\eps-
      \phi^\eps(t_n)}{\eps}\right\|_{L^\infty}, 
\end{equation*}
and Corollary~\ref{cor:wave} is a direct consequence of
Theorem~\ref{theo:main}. 
\subsection{Proof of Corollary~\ref{cor:CVquad}}
\label{sec:cor2}

Corollary~\ref{cor:CVquad} also stems directly from
Theorem~\ref{theo:main}, by resuming the same computations as in the
proof of Corollary~\ref{cor:quad-loc}.

\section{Weakly nonlinear regime}
\label{sec:wnlgo}
 We now consider \eqref{eq:nls} in the case $\alpha\ge 1$, under
 Assumption~\ref{hyp:f} on the nonlinearity. In view of the formal
 computations presented in Section~\ref{sec:wkb}, the case $\alpha=1$
 can be considered as the only interesting one, since no nonlinear
 effect is expected at leading order when $\alpha>1$. Since it is
 possible to treat both cases at once, we take advantage of
 this opportunity. 
\smallbreak

The analysis in the case $\alpha\ge 1$ being quite easier than in the
case $\alpha=0$ (even under Assumption~\ref{hyp:r3}, which is weaker
than Assumption~\ref{hyp:f}), we shall simply underline the
modifications to be made in order to prove
Proposition~\ref{prop:wnlgo} by following the same steps as in the
proof of Theorem~\ref{theo:main}. 
\smallbreak

To characterize the exact flow in terms of WKB states,
\eqref{eq:grenier} is replaced by
\begin{equation}
  \label{eq:grenier-weak}
  \left\{
\begin{aligned}
      &\d_t \phi^\eps+\frac{1}{2}|\nabla \phi^\eps|^2 =0 ;\quad
      \phi^\eps_{\mid t=0} =  
\phi_0,\\
&\d_t a^\eps +\nabla \phi^\eps\cdot \nabla a^\eps
+\frac{1}{2}a^\eps\Delta \phi^\eps = i\frac{\eps}{2}\Delta
a^\eps-i\eps^{\alpha-1}f\(|a^\eps|^2\)a^\eps;\quad a^\eps_{\mid t=0}=a_0. 
    \end{aligned}
\right.
\end{equation}
Thanks to the assumption $\alpha\ge 1$, the last term in the equation
for $a^\eps$ is not singular as $\eps\to 0$. More importantly, this is
no longer a coupled system: the first equation is an eikonal equation,
which we have analyzed in Section~\ref{sec:lin}. 
\smallbreak

In the numerical scheme, the operator $\mathcal X_\eps^t$,
corresponding to the free Schr\"odinger flow, is the same as before,
and analyzed in Section~\ref{sec:lin}. On the other hand the operator
$\mathcal Y_\eps^t$ can be modified, since the nonlinearity does not
affect the rapid oscillations (as can be seen also from
\eqref{eq:grenier-weak}). We recall that we now consider
\begin{equation*}
  \left\{ 
 \begin{aligned}
    & \d_t \phi^\eps =0;\quad \phi^\eps_{\mid
      t=0}=\phi_0^\eps,\\
&\d_t a^\eps=-i\eps^{\alpha-1}f\(|a^\eps|^2\)a^\eps;\quad a^\eps_{\mid
  t=0}=a_0^\eps.  
  \end{aligned}
\right.
\end{equation*}
We see that the possible loss of regularity pointed out in
Remark~\ref{rem:loss} is not present here, since the regularity of
$\phi^\eps$ is not affected by the regularity of $a^\eps$. Also,
working with $a^\eps$ in $H^s$ for $s>d/2$ ensures that the analysis
of Section~\ref{sec:lin} can easily be adapted under
Assumption~\ref{hyp:f}, since $a^\eps(t) =
a_0^\eps\exp\(-i\eps^{\alpha-1}t f(|a_0^\eps|^2)\)$.  
\smallbreak

The main modification in the analysis concerns the local error
estimate, since the statement of Lemma~\ref{lem:crochet} must be
revised. The operator $A$ remains unchanged, and the operator $B$
becomes
\begin{equation*}
  B  \begin{pmatrix}
    \phi\\
a
  \end{pmatrix}
=
\begin{pmatrix}
  0\\
  -i\eps^{\alpha-1}f\(|a|^2\)a
\end{pmatrix}.
\end{equation*}
We compute successively
\begin{equation*}
  B' \begin{pmatrix}
    \phi\\
a
  \end{pmatrix}
\begin{pmatrix}
    \varphi\\
b
  \end{pmatrix}
=
-i\eps^{\alpha-1}\begin{pmatrix}
 0 \\
 2f_1\(\RE \(\overline a
 b\)\)a+f_1\(|a|^2\)b+2
 f'_2\(|a|^2\)\RE \(\overline a b\)
\end{pmatrix},
\end{equation*}
and
\begin{equation*}
  [A,B]\begin{pmatrix}
    \phi\\
a
  \end{pmatrix} =
 i\eps^{\alpha-1} \begin{pmatrix}
   0\\
 F(\phi,a)
  \end{pmatrix},
\end{equation*}
where
\begin{align*}
  F(\phi,a) &= \nabla\phi\cdot \nabla \(f\(|a|^2\)a\)
  +\frac{1}{2}f\(|a|^2\)a\Delta \phi - i\frac{\eps}{2}\Delta
  \(f\(|a|^2\)a\)\\
&\quad -a \DIV f_1\(|a|^2\nabla \phi\) -\eps a \DIV f_1
\(\IM\(\overline a \nabla a\)\) \\
&\quad-f'_2\(|a|^2\) \DIV \(|a|^2\nabla
\phi+\eps \IM\(\overline a \nabla a\)\) .
\end{align*}
The main point to notice is that if $s>d/2+2$, then $F$ maps
$H^{s}\times H^s$ to $H^{s-2}$. Proposition~\ref{prop:wnlgo} then
follows by resuming the same steps as in the proof of
Proposition~\ref{prop:bound}.

\appendix

\section{Linear Schr\"odinger equation}
\label{sec:linear}
Consider the (linear) Schr\"odinger equation with a potential,
\begin{equation}
  \label{eq:linearSch}
  i\eps\d_t u^\eps +\frac{\eps^2}{2}\Delta u^\eps = V u^\eps;\quad
  u^\eps_{\mid t=0}= u_0^\eps,
\end{equation}
with $V=V(t,x)\in \R$. We assume that $V$ grows at
most quadratically in space:
\begin{hyp}\label{hyp:V}
  $V\in L^\infty_{\rm loc}([0,\infty)\times \R^d)$ is real-valued, and
  smooth with respect 
  to the space variable: for (almost) all $t\ge 0$, $x\mapsto V(t,x)$
  is a $C^\infty$ map. Moreover, it is at most quadratic in space: 
  \begin{equation*}
    \forall \alpha \in \N^d, \ |\alpha|\ge 2, \ \forall T>0,\quad
    \d_x^\alpha V\in L^\infty([0,T]\times \R^d). 
  \end{equation*}
In addition, $t\mapsto V(t,0)$ belongs to $L^\infty_{\rm loc}([0,\infty))$.
\end{hyp}
Then for $u_0^\eps \in L^2(\R^d)$,
\eqref{eq:linearSch} has a unique solution $u^\eps\in
C([0,\infty);L^2(\R^d))$, and its $L^2$-norm is conserved,
$\|u^\eps(t)\|_{L^2}=\|u_0^\eps\|_{L^2}$ for all $t\ge 0$; see
e.g. \cite{Fujiwara}. The following result is standard in
semi-classical analysis (see e.g. \cite{Robert}). We sketch the proof for
completeness. 
\begin{proposition}\label{prop:linearbound}
  Let $k\in \N$, $V$ satisfying Assumption~\ref{hyp:V}, and $u_0^\eps\in
  L^2(\R^d)$. Suppose in addition  that $u_0^\eps$ satisfies
  \begin{equation}
    \label{eq:sigmak}
\|u_0^\eps\|_{\Sigma_\eps^k}:=    \sup_{0<\eps\le 1}\(
\|u_0^\eps\|_{L^2} + \lVert \lvert x\rvert^k 
    u_0^\eps\rVert_{L^2} +\lVert \lvert \eps \nabla\rvert^k
    u_0^\eps\rVert_{L^2}\)<\infty.  
  \end{equation}
Then for all $T>0$, the solution to \eqref{eq:linearSch} satisfies
  \begin{equation*}
    \sup_{0<\eps\le 1}\sup_{t\in [0,T]}\( \|u^\eps(t)\|_{L^2} + \lVert
    \lvert x\rvert^k 
    u^\eps(t)\rVert_{L^2} +\lVert \lvert \eps \nabla\rvert^k
    u^\eps(t)\rVert_{L^2}\)<\infty.  
  \end{equation*}
\end{proposition}
\begin{proof}
  The key point is that the functions $\eps\nabla u^\eps$ and $x u^\eps$
  satisfy a closed system of estimates. Indeed, $\eps\nabla$ does not
  commute with the equation, and $\eps\nabla u^\eps$ satisfies
  \begin{equation*}
    i\eps\d_t \(\eps\nabla u^\eps\) +\frac{\eps^2}{2}\Delta \(\eps\nabla
    u^\eps\) = V \eps\nabla  u^\eps +\(\eps\nabla V\)u^\eps.
  \end{equation*}
Similarly, 
  \begin{equation*}
  i\eps\d_t \(x u^\eps\) +\frac{\eps^2}{2}\Delta \(x
    u^\eps\) = V x  u^\eps +\eps^2 \nabla u^\eps.  
  \end{equation*}
The standard $L^2$ estimate then yields
\begin{align*}
  &\|\eps\nabla u^\eps(t)\|_{L^2}\le \|\eps\nabla u_0^\eps\|_{L^2} +
  \int_0^t \|(\nabla V)u^\eps(\tau)\|_{L^2}d\tau,\\
 &\|x u^\eps(t)\|_{L^2}\le \|x u_0^\eps\|_{L^2} +
  \int_0^t \|\eps \nabla u^\eps(\tau)\|_{L^2}d\tau.
\end{align*}
Now under Assumption~\ref{hyp:V}, for $T>0$ fixed, we have the
pointwise estimate
\begin{equation*}
  \left\lvert\nabla V(\tau,x)u^\eps(\tau,x)\right\rvert \le C(T)
  \(1+|x|\)\left\lvert u^\eps(\tau,x)\right\rvert, \quad 0\le \tau\le T.
\end{equation*}
Recalling that the $L^2$-norm of $u^\eps$ is bounded, Gronwall lemma,
applied to 
\begin{equation*}
  y(t) = \|\eps\nabla u^\eps(t)\|_{L^2}+\|x u^\eps(t)\|_{L^2},
\end{equation*}
yields the proposition in the case $k=1$. The general case follows by
induction. 
\end{proof}
\begin{example}
 If $u_0^\eps$ is of WKB type \eqref{eq:ci}, or more generally
  \eqref{eq:bkw}, with $\phi_0$ at most quadratic (in the sense of
  Assumption~\ref{hyp:V}), and $a_0\in H^k\cap \F (H^k)$, then the
  above assumptions are fulfilled. Note however that
  Proposition~\ref{prop:linearbound} is valid for all time, and in
  particular after the formation of caustics, if any. 
\end{example}
\begin{example}
  If
  \begin{equation*}
    u_0^\eps(x) = \frac{1}{\eps^{\theta
        d/2}}a_0\(\frac{x-q}{\eps^\theta}\)e^{i(x-q)\cdot p/\eps}, 
  \end{equation*}
with $q,p\in \R^d$, $\theta\in [0,1]$, and $a_0\in \Sch(\R^d)$, then
again, Proposition~\ref{prop:linearbound} is valid for all time. If
$\theta=0$, this datum is a particular WKB datum (with a linear
phase). If 
$\theta=1/2$, this means that an initial coherent state is
considered (see e.g. \cite{CoRO97}). If $\theta=1$, the initial datum
is concentrating at point $q$, 
which corresponds to a caustic reduced to one point (focal point; see
\cite{CaBook}).  
\end{example}
Recall that if the splitting operators are defined by
\begin{equation*}
  A = i\frac{\eps}{2} \Delta;\quad B = -\frac{i}{\eps}V,
\end{equation*}
then their Lie commutator is given by 
\begin{equation*}
  [A,B] = \nabla V\cdot \nabla + \frac{1}{2}\Delta V.
\end{equation*}
With  the norm
$\|u\|_{\Sigma^2_\eps}$  defined in
\eqref{eq:sigmak},  note the control
\begin{equation*}
  \|\eps \nabla V\cdot \nabla
u\|_{L^2}\lesssim 
\|u\|_{\Sigma^2_\eps},
\end{equation*}
which follows from Assumption~\ref{hyp:V}.
By working with the norm $\|u\|_{\Sigma^2_\eps}$, rather than with the
norm $\|u\|_{H^1_\eps}$  
defined in Section~\ref{sec:split}, and
used in \cite{BJM1,DeTh10}, the following
result is a direct consequence of \cite{DeTh10} and 
Proposition~\ref{prop:linearbound}:
\begin{proposition}
 Let $d\ge 1$, and $V$ satisfying Assumption~\ref{hyp:V}. Suppose that
 $\|u_0^\eps\|_{\Sigma_\eps^2}<\infty$. Then for all $T>0$, there
 exist $C,c_0$ independent of $\eps\in (0,1]$
such that for 
  all $\Dt\in (0,c_0]$, for all $n\in \N$ such that $n\Dt\in
  [0,T]$, 
\begin{equation*}
 \left\| \(Z^{\Dt}_{\eps}\)^n u_0^\eps
   -S^{t_n}_\eps
      u_0^\eps\right\|_{L^2(\R^d)}\le C \frac{\Dt}{\eps},
 \end{equation*}
where $S_\eps^tu_0^\eps = u^\eps(t)$ in \eqref{eq:linearSch}, and
$Z_\eps^t = e^{tB}e^{tA}$. 
\end{proposition}

\subsection*{Acknowledgements}
The author is grateful to Christophe Besse and St\'ephane
Descombes for precious references, and to the referees and to 
Luca Dieci as an editor, for their
constructive remarks.

\end{document}